\documentclass[11pt,reqno]{amsart}

\usepackage[usenames,dvipsnames]{xcolor}
\RequirePackage[OT1]{fontenc}

\usepackage[colorlinks,citecolor=blue,urlcolor=blue,linkcolor=blue,linktocpage=true]{hyperref}

\usepackage[]{amsmath}
\usepackage{amssymb,amsthm,amsfonts,amsbsy,latexsym,amsxtra}

\usepackage{graphicx}

\usepackage{todonotes}
\usepackage{appendix}
\usepackage{url}

\allowdisplaybreaks[3]

\usepackage[usenames,dvipsnames]{xcolor}            
\usepackage{graphicx}
\usepackage{bbm,bm}
\usepackage{comment}
\usepackage{mathtools}
\usepackage{placeins}
\usepackage[labelfont={sl}]{caption}
\usepackage[labelfont={normalfont}]{subcaption}

\usepackage{amssymb,amsfonts,amsmath,amsthm,amscd,dsfont,mathrsfs}

\usepackage{url}

\usepackage{enumerate}

\numberwithin{equation}{section}

\newtheorem{theorem}{Theorem}[section]

\newtheorem{lemma}[theorem]{Lemma}

\theoremstyle{definition}

\newtheorem{example}[theorem]{Example}

\newcommand{\E}{\mathbf{E}}
\newcommand{\R}{\mathbb{R}}

\newcommand{\M}{\mu}
\newcommand{\sP}{\mathcal{P}}
\newcommand{\N}{\mathbb{N}}

\newcommand{\eps}{\varepsilon}
\renewcommand{\P}{\mathbf{P}}
\newcommand{\Prob}[1]{\mathbf P\left\{#1\right\}}

\renewcommand{\emptyset}{\varnothing}
\newcommand{\XX}{\mathbb{X}}
\newcommand{\ZZ}{\mathbb{Z}}
\newcommand{\QQ}{\mathbb{Q}}
\newcommand{\F}{\mathcal{F}}
\newcommand{\Nb}{\mathbf{N}}
\newcommand{\e}{\varepsilon}

\newcommand{\diff}{{\,\mathrm d}}
\newcommand{\ldiff}{{\mathrm d}}

\renewcommand{\preceq}{\prec}
\renewcommand{\succeq}{\succ}

\DeclareMathOperator{\Var}{Var}

\allowdisplaybreaks[3]

\usepackage{color}

\newtheorem*{theorem*}{Theorem}

\begin{document}

\title[Region-Stabilizing Scores]{Gaussian Approximation for Sums
  of Region-Stabilizing Scores}

\author[C. Bhattacharjee]{Chinmoy Bhattacharjee}
\address{Department of Mathematics, University Luxembourg, Luxembourg}
\email{chinmoy.bhattacharjee@uni.lu}

\author[I. Molchanov]{Ilya Molchanov}
\address{Institut f\"ur Mathematische Statistik und Versicherungslehre, University of Bern, Switzerland}
\email{ilya.molchanov@stat.unibe.ch}

\date{\today}

\thanks{IM was supported by the Swiss National Science Foundation Grant
  No.\ 200021\_175584}


\begin{abstract}
  We consider the Gaussian approximation for functionals of a Poisson
  process that are expressible as sums of region-stabilizing
  (determined by the points of the process within some specified
  regions) score functions and provide a bound on the rate of
  convergence in the Wasserstein and the Kolmogorov
  distances. While such results have previously been shown in
  Lachi\`eze-Rey, Schulte and Yukich (2019), we extend the applicability by relaxing some
  conditions assumed there and provide further insight into the
  results. This is achieved by working with stabilization regions that
  may differ from balls of random radii commonly used in the
  literature concerning stabilizing functionals. We also allow for
  non-diffuse intensity measures and unbounded scores, which are
  useful in some applications. As our main application, we consider
  the Gaussian approximation of number of minimal points in a
  homogeneous Poisson process in $[0,1]^d$ with $d \ge 2$, and provide a presumably
  optimal rate of convergence.
\end{abstract}

\maketitle

\section{Introduction}

Let $(\XX,\F)$ be a Borel space and let $\QQ$ be a $\sigma$-finite
measure on $(\XX, \F)$.  For $s \ge 1$, let $\sP_s$ denote a Poisson
process with intensity measure $s\QQ$.  Our main object of study is
the sum of score functions $(\xi_s)_{s \ge 1}$ given by
\begin{equation}
	\label{eq:hs}
	H_s=H_s(\sP_s):= \sum_{x \in \sP_s} \xi_s(x,\sP_s), \quad s \ge 1,
\end{equation}
when the sum converges.
While $H_s$ is a functional of the whole point process, this
representation implicitly assumes that the functional can be
decomposed as a sum of local contributions at each point
$x \in \sP_s$.  Indeed, in the vast literature on limit theorems for
sums of score functions over points in a Poisson process (see, e.g.,
\cite{pen:yuk03,pen:yuk05,sch10}), it is usually assumed that the
score function at a point $x$ depends on the whole point process only
through the set of its points within some small (random) distance to
$x$, prohibiting any long-range interactions.  Conditions like
exponential decay of the tail distribution of this distance, so-called
`radius of stabilization', and bounds on certain moments of the score
functions are crucial to derive a quantitative central limit
theorem. The idea of using \textit{stabilization} for studying limit
theorems started with the works \cite{PY01,pen:yuk03}. Subsequently,
important further works advanced such quantitative results for the
Gaussian approximation of stabilizing functionals, see, e.g.,
\cite{BX06,pen:yuk05,Yu15}. But all these results provided bounds that
had an extraneous logarithmic factor multiplied to the inverse of the
square root of the variance.  The results in this area culminated in
\cite{LPS16}, where, using Malliavin-Stein approach, this logarithmic
factor was removed, and further in \cite{LSY19}, providing presumably
optimal rates and ready-to-use conditions illustrated with numerous
applications.

The comparative simplicity of the bounds provided in \cite{LSY19}
comes at the cost of assuming a few conditions on the underlying space
and the score functions. Even though these conditions are satisfied in
many important examples as demonstrated therein, they are not
applicable in some cases, especially, in examples exhibiting
long-range interactions. A notable example is the number of minimal
(or Pareto optimal) points in $\sP_s$ restricted to the unit cube
$[0,1]^d$, $d \ge 2$. This example violates all existing stabilization
conditions usually assumed in the context of quantitative limit
theorems. In particular, the appearance of stabilization regions that can be
arbitrarily thin and long makes the radius of stabilization too large
to obtain a meaningful bound using results from \cite{LSY19}. As a
result, \cite{LSY19} could
only manage to handle (in the problem of counting maximal points, which is equidistributed as the number of minimal points) a modified setting, by replacing the cube with a
domain of the form $\{x \in [0,\infty)^d : F(x) \le 1\}$, where
$F : [0,\infty)^d \to [0,\infty)$ is strictly increasing in each
coordinate with $F(0)<1$, is continuously differentiable, and has
continuous partial derivatives that are bounded away from zero and
infinity. Even though one can define a function $F$ to obtain a domain
that is arbitrarily close to the cube, the behavior of the number of
maximal points is very sensitive to small changes in the shape of the
domain: while the variance of $H_s$ is of the order of $s^{(d-1)/d}$
in the setting of \cite{LSY19}, its order becomes
$\log^{d-1} s$ in the case of the cube, see \cite{bai:dev:hwan:05}.

The main aim of this paper is to develop a more versatile notion of
stabilization that enables us to handle various examples with long-range interactions, most notably
the example of minimal points in the cube. We achieve this by
generalizing the concept of stabilization radius to allow for regions
of arbitrary shape, that is, by replacing balls of random radii with
general sets, called stabilization regions. It is unlikely to achieve
this by amending the metric on the carrier space, since the shape of
these stabilization regions may be random and depend heavily on the
reference point, and also since the stabilization region may be
empty. The only additional condition we assume is that the
stabilization region is monotonically decreasing in the point
configuration, which is a natural condition satisfied by all common
examples.

In addition, we also extend the results to non-diffuse intensity
measures and to score functions with non-uniform bounds on their
moments. The extension to non-diffuse intensity measures results from
getting rid of some regularity assumption on $\QQ$ imposed in
\cite{LSY19}. This makes it possible to handle examples with multiple
points at deterministic locations, like Poisson processes on
lattices. The extension to scores with unbounded moments is crucial in
examples where the score functions are not simple indicators but
rather involve unbounded weight functions, or when the intensity
measure is infinite. 
Such an extension is a byproduct
of our generalization of \cite[Theorem~6.1]{LPS16}, which involves
non-uniform bounds on the $(4+p)$-th moment of the first order difference operator for some $p >0$, see
Theorem~\ref{thm:Main}. We present two examples concerning isolated
points in the two-dimensional integer lattice and a random geometric
graph in $\R^d$, $d \ge 2$, to demonstrate further applications of our
general bounds. Apart from the fact that our approach is more
versatile than that of \cite{LSY19}, to the best of our knowledge,
working with general monotonically decreasing stabilization sets is
new in the relevant literature and thus our work opens a new direction
of investigation. It should be noted that the very comprehensive
setting in \cite{LSY19} also covers the cases of Poisson processes
with marks, as well as the setting of binomial processes.  Our
results can be extended to these settings by adapting the scheme
elaborated in \cite{LSY19} to our approach relying on stabilization
regions. Indeed, Theorem~4.2 in \cite{LSY19} providing a bound on
Gaussian approximation for functionals of a binomial process can be
modified to the setting with a non-uniformly bounded $(4+p)$-th moment
of the difference operator in the same way we modify Theorem~6.1 in
\cite{LPS16} in our Theorem~\ref{thm:Main}. Once this key step is
achieved, one can follow our line of argument to obtain a result
paralleling our Theorem~\ref{thm:KolBd} for binomial processes.

Let us now explicitly describe our setup. For a Borel space
$(\XX,\F)$, denote by $\Nb$ the family of $\sigma$-finite counting
measures $\M$ on $\XX$ equipped with the smallest $\sigma$-algebra
$\mathscr{N}$ such that the maps $\M \mapsto \M(A)$ are measurable for
all $A \in \F$. We write $x\in\M$ if $\M(\{x\})\geq1$.
Denote by $0$ the zero counting measure. Further, $\M_A$ denotes the
restriction of $\M$ onto the set $A\in\F$, and $\delta_x$ is the Dirac
measure at $x\in\XX$. For $\M_1,\M_2\in\Nb$, we write $\M_1\leq\M_2$
if the difference $\M_2-\M_1$ is non-negative.

For each $s\geq 1$, a \emph{score function} $\xi_s$ associates to each
pair $(x,\M)$ with $x\in\XX$ and $\M\in\Nb$, a real number
$\xi_s(x,\M)$. Throughout, we assume that the function
$\xi_s: \XX \times \Nb \to \R$ is measurable with respect to the
product $\sigma$-algebra $\F\otimes\mathscr{N}$ for all $s\geq1$.

With $H_s$ as in \eqref{eq:hs}, our aim is to find an upper bound on
the distance between the distributions of the normalized sum of scores
$(H_s-\E H_s)/\sqrt{{\rm Var}\,H_s}$ and a standard normal random
variable $N$ in an appropriate distance. We consider two very commonly
used distances, namely, the Wasserstein and the Kolmogorov
distances. The Wasserstein distance between (the distributions of) real-valued random
variables $X$ and $Y$ is given by
$$
d_W(X,Y):= \sup_{h \in \operatorname{Lip}_1} |\E\; h(X) - \E \; h(Y)|,
$$
where $\operatorname{Lip}_1$ denotes the class of all Lipschitz
functions $h: \R \to \R$ with Lipschitz constant at most one.  The
Kolmogorov distance between $X$ and $Y$ is defined by taking the test
functions to be indicators of half-lines, and is given by
$$
d_K(X,Y):= \sup_{t \in \R} |\Prob{X \le t} - \Prob{Y \le t}|.
$$

Following \cite{LSY19}, a score function stabilizes if $\xi_s(x,\M)$
remains unaffected when the configuration $\M$ is altered outside a
ball of radius $r_x=r_x(\mu)$ (the radius of stabilization) centered at
$x$. For this, it is assumed that $\XX$ is a semimetric space and
$\QQ$ satisfies a technical condition concerning the $\QQ$-content of
an annulus in the space $\XX$, which in particular implies that $\QQ$
is diffuse.  In \cite{LSY19}, under an exponential decay condition on
the tail distribution of the stabilization radius $r_x$ as
$s \to \infty$ and assuming that the $(4+p)$-th moment of the score
function at $x$ is uniformly bounded by a constant for all $s\geq1$
and $x\in\XX$ for some $p \in (0,1]$, a universal bound on the
Wasserstein and Kolmogorov distances between the normalized sum of
scores and $N$ was derived.

The setting of stabilization regions as balls centered at $x\in\sP_s$
with radius $r_x$ can be thought of as a special case of a more general
concept of stabilization regions which are sets depending on $x$ and
the Poisson process.  Indeed, in some examples, it is not optimal to
assume that the stabilization region is a ball. The region can be made
substantially smaller if it is allowed to be of a general
shape. Adjusting the theory to deal with such stabilization regions is
the main contribution of our work. Our general setting of non-spherical
stabilization regions also eliminates the need of extra technical
assumptions on the intensity measure imposed in \cite{LSY19}. As an
illustration, we show how to handle the example of minimal points in
the unit cube, which does not fit into the framework of \cite{LSY19}.
We also allow for multiple points and for a non-uniform bound on the
$(4+p)$-th moment of the score functions, which is particularly
important in examples involving infinite intensity measures, like
stationary Poisson processes. Apart from examples presented in the
current paper, further applications of our method has been elaborated
in \cite{bhat-mol21}, where a quantitative central limit theorem is obtained for
functionals of growth processes that result in generalized Johnson-Mehl tessellations, and in \cite{Bha21},
where such a result is obtained in the context of minimal directed
spanning trees in dimensions three and higher, respectively.

\section{Notation and main results}
\label{sec:notat-main-results}

Throughout the paper, for $s \ge 1$, we consider a
$\F\otimes\mathscr{N}$-measurable score function $\xi_s(x,\M)$. Assume
that if $\xi_s(x, \M_1)=\xi_s(x, \M_2)$ for some $\M_1,\M_2 \in \Nb$
with $0\neq \M_1\leq \M_2$, then
\begin{equation}
	\label{eq:ximon}
	\xi_s(x, \M_1)=\xi_s(x, \M') \quad 
	\text{for all} \, \M'\in\Nb \; \text{ with } \; \M_1\leq \M'\leq \M_2.
\end{equation}
This is a natural condition to expect for any reasonably well-behaved
score function.  We will need a few more assumptions on the score
functions. The first assumption is a generalization of the concept of
stabilization radius.
\begin{enumerate}
	\item[(A1)] \textit{Stabilization region:} For all $s \ge 1$,
	there exists a map $R_s$ from $\{(x,\M)\in\XX\times \Nb:x\in\M\}$
	to $\F$ such that
	\begin{enumerate}[(1)]
		\item[(A1.1)] the set
		\begin{displaymath}
			\{(x,y_1, y_2,\mu): \{y_1,y_2\}\subseteq  R_s(x, \mu+\delta_x)\}
		\end{displaymath}
		is measurable with respect to the product $\sigma$-algebra on
		$\XX^3\times\Nb$,
		\item[(A1.2)] the map $R_s$ is monotonically decreasing in the second argument,
		i.e.\
		$$
		R_s(x,\M_1) \supseteq R_s(x,\M_2),
		\quad \M_1 \leq \M_2,\; x\in\M_1,
		$$ 
		\item[(A1.3)] for all $\mu\in\Nb$ and $x\in\mu$, $\mu_{R_s(x,\mu)} \neq 0$
		implies $(\mu+\delta_y)_{R_s(x,\mu +\delta_y)} \neq 0$ for all
		$y \notin R_s(x,\mu)$,  
		\item[(A1.4)] for all $\M\in\Nb$ and $x\in\M$,
		\begin{displaymath}
			\xi_{s}\big(x,\M\big)
			=\xi_{s}\big(x,\M_{R_{s}(x,\M)}\big).
		\end{displaymath}
	\end{enumerate}
\end{enumerate}

By taking the intersection of the set from (A1.1) with the set
$\{(x,y,y,\mu):\mu\in\Nb\} \subseteq \XX^3 \times \Nb$ (which is also measurable) and then
applying the bijective projection on $\Nb$ we see that
\begin{equation}
	\label{eq:1}
	\{\mu \in \Nb : y\in R_s(x,\mu+\delta_x)\}\in\mathscr{N}
\end{equation}
for all $(x,y)\in\XX^2$. Furthermore, Fubini's theorem implies that
\begin{equation}
	\label{eq:2}
	\Prob{y\in R_s(x, \sP_s+\delta_x)}\, \text{ and } \,
	\Prob{\{y_1, y_2\}\subseteq  R_s(x, \sP_s+\delta_x)}
\end{equation}
are Lebesgue measurable functions of $(x,y) \in \XX^2$ and
$(x,y_1,y_2) \in \XX^3$, respectively. Even though, assumption (A1.1)
is sufficient for our result, it is indeed enough to assume
\eqref{eq:1} and \eqref{eq:2}. Thus, when simpler, we will verify the
conditions \eqref{eq:1} and \eqref{eq:2} instead of (A1.1).

Note that (A1) holds trivially if one takes $R_s$ to be identically
equal to the whole space $\XX$. If (A1) holds with a non-trivial
$R_s$, then the score function is called
\emph{region-stabilizing}. Also note that a condition like
\cite[Eq.~(2.3)]{LSY19}, requiring stabilization with 7 additional
points, trivially holds in our set up due to the monotonicity
assumption (A1.2) and \eqref{eq:ximon}.


We also assume the standard $(4+p)$-th moment condition, stated here
in terms of the norm for notational simplicity. In the following,
$\|\cdot\|_{4+p}$ denotes the $L^{4+p}$-norm.
\begin{enumerate}
	\item[(A2)] \textit{$L^{4+p}$-norm:} There exists a $p \in (0,1]$
	such that, for all $\mu\in\Nb$ with $\mu(\XX) \le 7$,
	\begin{displaymath}
		\Big\|\xi_{s}\big(x, \sP_{s}+\delta_x+\mu\big)\Big\|_{4+p}
		\leq M_{s,p}(x), \quad s\geq1,\; x \in \XX,
	\end{displaymath}
	where $M_{s,p} : \XX \to \R$, $s\geq1$, are measurable
	functions.
\end{enumerate}

If the score function is an indicator random variable, Condition (A2)
is trivially satisfied with $M_{s,p}\equiv 1$ for any $p \in (0,1]$
and $s\geq1$. For notational convenience, in the sequel we will write $M_s$ instead of $M_{s,p}$, and generally drop $p$ from all subscripts.

Let $r_{s}: \XX \times \XX \to [0,\infty]$ be a measurable function
such that
\begin{equation}
	\label{eq:Rs}
	\Prob{y \in R_s(x, \sP_s +\delta_x)}
	\le  e^{-r_{s}(x,y)}, \quad x, y \in \XX.
\end{equation}
For the following it is essential that $r_s$ does not vanish, and then
\eqref{eq:Rs} becomes an analog of the usual exponential stabilization
condition from \cite{LSY19}. Note that we allow $r_s$ to be infinite and the probability in
\eqref{eq:Rs} is well defined due to assumption (A1.1). 

For
$x_1,x_2 \in \XX$, denote
\begin{equation}
	\label{eq:g2s}
	q_{s}(x_1,x_2):=s \int_\XX \P\Big\{\{x_1,x_2\} 
	\subseteq R_s\big(z, \sP_s +\delta_z\big)\Big\} \;\QQ(\ldiff z),
\end{equation}
noticing that the probability in the integral is well defined and
$q_s$ is measurable due to Fubini's theorem and \eqref{eq:2}.

For $p \in (0,1]$ as in (A2) and $\zeta:=p/(40+10p)$, let
\begin{gather}
	\label{eq:g}
	g_{s}(y) :=s \int_{\XX} e^{-\zeta r_{s}(x, y)} \;\QQ(\ldiff x), \quad h_s(y):=s \int_{\XX} M_{s}(x)^{4+p/2}e^{-\zeta r_{s}(x, y)}\;\QQ(\ldiff x),\\ 
	\label{eq:g5}
	G_s(y) := \widetilde{M}_{s}(y) + \tilde h_s(y) \big(1+g_{s}(y)^4\big), \quad y\in\XX,
\end{gather}
where for $y \in \XX$,
$$
\widetilde M_s(y):=\max\{M_s(y)^2,M_s(y)^4\} \quad \text{and} \quad \tilde h_s(y):=\max\{h_s(y)^{2/(4+p/2)}, h_s(y)^{4/(4+p/2)}\}.$$ 
For
$\alpha>0$, let
\begin{equation}
	\label{eq:fa}
	f_\alpha(y):=f_\alpha^{(1)}(y)+f_\alpha^{(2)}(y)+f_\alpha^{(3)}(y),
	\quad y\in\XX,
\end{equation}
where 
\begin{align}
	\label{eq:fal}
	f_\alpha^{(1)}(y)&:=s \int_\XX G_s(x) e^{- \alpha r_{s}(x,y)}
	\;\QQ(\ldiff x), \notag \\
	f_\alpha^{(2)} (y)&:=s \int_\XX G_s(x) e^{- \alpha r_{s}(y,x)} 
	\;\QQ(\ldiff x), \nonumber\\
	f_\alpha^{(3)}(y)&:=s \int_{\XX} G_s(x) q_{s}(x,y)^\alpha \;\QQ(\ldiff x).
\end{align}
Finally, define the function
\begin{equation}
	\label{eq:p}
	\kappa_s(x):= \Prob{\xi_{s}(x, \sP_{s}+\delta_x) \neq 0},\quad
	x\in\XX. 
\end{equation}

Our main result is the following abstract theorem, which generalizes
Theorem~2.1(a) in \cite{LSY19}. 
For an integrable function $f : \XX \to \R$, denote
$\QQ f:=\int_\XX f(x) \QQ(\ldiff x)$. 

\begin{theorem}\label{thm:KolBd}
	Assume that $(\xi_s)_{s \ge 1}$ satisfy conditions (A1), (A2) and
	let $H_s$ be as in \eqref{eq:hs}. Then, for $p$ as in (A2) and
	$\beta:=p /(32+4 p)$, 
	\begin{align*}
		d_{W}\left(\frac{H_s-\E H_s}{\sqrt{\Var H_s}},  N\right) 
		&\leq C \Bigg[\frac{\sqrt{s  \QQ f_\beta^2}}{\Var H_s}
		+\frac{ s\QQ ((\kappa_s+g_{s})^{2\beta}G_s)}{(\Var H_s)^{3/2}}\Bigg],
	\end{align*}
	and
	\begin{align*}
		d_{K}\left(\frac{H_s-\E H_s}{\sqrt{\Var H_s}},
		N\right) 
		&\leq C \Bigg[\frac{\sqrt{s  \QQ f_\beta^2}
			+ \sqrt{s  \QQ f_{2\beta}}}{\Var H_s}
		+\frac{\sqrt{ s\QQ ((\kappa_s+g_{s})^{2\beta}G_s)}}{\Var H_s}
		+\frac{ s\QQ ((\kappa_s+g_{s})^{2\beta} G_s)}{(\Var H_s)^{3/2}}\\
		&\qquad\qquad +\frac{( s\QQ ((\kappa_s+g_{s})^{2\beta} G_s))^{5/4}
			+ ( s\QQ ((\kappa_s+g_{s})^{2\beta} G_s))^{3/2}}{(\Var H_s)^{2}}\Bigg]
	\end{align*}
	for all $s\geq1$, where $N$ is a standard normal random variable and
	$C \in (0,\infty)$ is a constant depending only on $p$.
\end{theorem} 

In order to obtain a useful bound, it is necessary that
$\QQ (\widetilde M_s\kappa_s)$ is finite. This is surely the case
if $\QQ$ is finite and $\widetilde M_s$ is bounded.

As an application of our abstract result, we consider an example
regarding \emph{minimal points} in a Poisson process. Let $\QQ$ be the
Lebesgue measure on $\XX:=[0,1]^d$, $d \ge 2$, and let $\sP_s$ be a
Poisson process with intensity $s \QQ$ for $s\geq1$. A point
$x \in \R^d$ is said to dominate a point $y \in \R^d$ if
$x-y\in\R_+^d\setminus\{0\}$. We write $x\succ y$, or equivalently,
$y \prec x$ if $x$ dominates $y$.  Points in $\sP_s$ that do not
dominate any other point in $\sP_s$ are called minimal (or Pareto
optimal) points of $\sP_s$. The interest in studying dominance and
number of minima and maxima is due to its numerous applications
related to multivariate records, e.g., in the analysis of linear
programming and in maxima-finding algorithms, see the references in
\cite{bai:dev:hwan:05} and \cite{fil:naim20}. In the following result,
we derive non-asymptotic bounds on the Wasserstein and Kolmogorov
distances between the normalized number of minimal points in $\sP_s$,
and a standard Gaussian random variable.

\begin{theorem}\label{thm:Pareto} 
	Let $\sP_s$ be a Poisson process on $[0,1]^d$ with intensity measure
	$s\QQ$ and $s \ge 1$, where $\QQ$ is the Lebesgue measure, and let
	\begin{equation}
		\label{eq:ParetoPoints}
		F_s:=\sum_{x \in \sP_s} \mathds{1}_{x \text{ is a minimal point in $\sP_s$}}.
	\end{equation}
	If $d\geq 2$, then 
	\begin{displaymath}
		\max \left\{d_W\left(\frac{F_s - \E F_s}{\sqrt{\Var
				F_s}},N\right),
		d_K\left(\frac{F_s - \E F_s}{\sqrt{\Var F_s}},N\right)\right\} \le
		\frac{C}{\log^{(d-1)/2} s}, \quad s\ge1, 
	\end{displaymath}
	for a constant $C>0$ depending only on the dimension $d$. In addition, the bound on the Kolmogorov distance is of optimal order, i.e., there exists a constant $0<C'\le C$ depending only on $d$ such that $d_K\left(\frac{F_s - \E F_s}{\sqrt{\Var F_s}},N\right) \ge C'/\log^{(d-1)/2} s$.
\end{theorem}

In the setting of binomial point process with $n \in \N$ i.i.d.\
points in the unit cube, \cite{bai:dev:hwan:05} showed that the
Wasserstein distance between the normalized number of minimal points
and the standard normal random variable is of the order
$(\log n)^{-(d-1)/2}(\log\log n)^{2d}$ using a log-transformation
trick first suggested in \cite{Ba20}, and, as a consequence, derived
the order $(\log n)^{-(d-1)/4}(\log\log n)^{d}$ for the Kolmogorov
distance. It is useful to note here that the variance of the number of
minimal points in the binomial case is of the order $\log^{d-1} n$,
see, e.g., \cite{bai:dev:hwan:05}, where the corresponding
computations in the Poisson case are also available. Hence, the
Wasserstein distance is of the order of the square root of the
variance multiplied by an extraneous logarithmic factor, which, as
mentioned before, has commonly appeared in such contexts. Furthermore,
the bound on the Kolmogorov distance is vastly suboptimal. Our result
in the Poisson setting substantially improves these rates to the
square root of the variance of $F_s$, which is optimal for the Kolmogorov distance and presumably optimal for the Wasserstein distance.

It should be noted that, in the example of Pareto optimal points, we
are working with a simple Poisson process and a finite intensity
measure $\QQ$. Further examples confirm that our abstract bound
applies also for Poisson processes with a non-diffuse or infinite
intensity measure $\QQ$. Note that for measures with infinite intensity,
\cite{LSY19} requires that the score function decays exponentially
with respect to the distance to some set $K$, and the bound in
Eq.~(2.10) therein becomes trivial if this set $K$ is the whole space
and $\QQ$ is infinite.

The rest of the paper is organized as follows. In
Section~\ref{sec:Pareto} we prove Theorem~\ref{thm:Pareto}. Section
\ref{sec:ex} provides two examples in settings, where either the
intensity measure is infinite and non-diffuse or the $(4+p)$-th
moments of the score functions are unbounded over the space $\XX$, and
provide bounds on the rate of convergences in the Wasserstein and the
Kolmogorov distances for Gaussian approximation of certain statistics
related to isolated points in these models. Finally, in
Section~\ref{sec:Proof} we prove Theorem~\ref{thm:KolBd} which relies
on a modified version of Theorem~6.1 in \cite{LPS16}, see Theorem~\ref{thm:Main}. The proof of the latter is presented in the Appendix.

\section{Number of minimal points in the hypercube}
\label{sec:Pareto}

In this section, we apply Theorem~\ref{thm:KolBd} to prove
Theorem~\ref{thm:Pareto} providing a quantitative limit theorem for
the number of minimal points in a Poisson process on the
hypercube. Throughout this section, $\QQ$ is taken to be the Lebesgue
measure on $\XX:=[0,1]^d$ with $d\in \N$, and $\sP_s$ is a Poisson
process on $\XX$ with intensity measure $s\QQ$ for $s \ge 1$. We omit
$\QQ$ in integrals and write $\diff x$ instead of $\QQ(\ldiff x)$. The
functional $F_s$ from \eqref{eq:ParetoPoints} can be expressed as in
\eqref{eq:hs} with the score functions
\begin{equation}
	\label{eq:xi}
	\xi_s(x,\M):=\mathds{1}_{x \text{ is a minimal point in
			$\M$}},\quad x\in\M, \; \M \in \Nb.  
\end{equation}
As a convention, we let $\xi_s(x,0)=0$. It is straightforward to see
that $(\xi_s)_{s \ge 1}$ satisfies \eqref{eq:ximon}. We will show that
conditions (A1) and (A2) also hold, so that Theorem~\ref{thm:KolBd} is
applicable.

For $x:=(x^{(1)},\dots,x^{(d)})\in \XX$, let
$[0,x]:=[0,x^{(1)}] \times\cdots\times [0,x^{(d)}]$, and denote the
volume of $[0,x]$ by
\begin{displaymath}
	|x|:=x^{(1)}\cdots x^{(d)}.
\end{displaymath}
Given a counting measure $\M\in\Nb$ and $x\in\M$, define the
stabilization region as
\begin{equation*}
	R_s(x,\M):= 
	\begin{cases}
		[0,x] & \mbox{if $\M([0,x]\setminus \{x\})=0$},\\
		\emptyset & \mbox{otherwise}.
	\end{cases} 
\end{equation*}
To begin with, we note here that the region $R_s$ can be the empty set
in our case, which rules out any possibility of it being represented
as a ball in some metric on the space $\XX$. Since for any
$x \in \XX$, the mapping
$\Nb \ni \mu \mapsto \mu([0,x]\setminus \{x\})$ is measurable, the
condition in \eqref{eq:1} follows. Next, it is easy to see that for
$x, y \in \XX$,
\begin{equation}\label{eq:Prob1}
	\Prob{y \in R_s(x,\sP_{s} + \delta_x)}=\mathds{1}_{x \succ y} e^{-s|x|},
\end{equation}
which is clearly measurable. Denote by $x_1\vee\dots\vee x_n$ the
coordinatewise maximum of $x_1,\dots,x_n \in \XX$, while
$x_1\wedge\dots\wedge x_n$ denotes their coordinatewise minimum. For
$x_1,x_2\in \XX$, notice that
$\{x_1,x_2\}\subseteq R_{s}(z,\sP_{s}+\delta_z)$ if and only if
$z\succ (x_1\vee x_2)$ and $[0,z] \setminus \{z\}$ has no points of
$\sP_s$. Thus
\begin{equation}\label{eq:Prob2}
	\Prob{\{x_1,x_2\}  \subseteq R_s(z,\sP_{s} + \delta_z)}
	=\mathds{1}_{z \succ x_1 \vee x_2} e^{-s|z|},
\end{equation}
which is also a measurable function of $(z,x_1, x_2)\in \XX^3$,
confirming \eqref{eq:2}. Clearly, $R_s$ is monotonically decreasing in
its second argument. It is straightforward to check (A1.3). Finally,
with $\xi_s$ as defined at \eqref{eq:xi}, it is easy to see that
(A1.4) is satisfied.  Furthermore, condition (A2) holds trivially with
$M_s\equiv 1$ for all $p \in (0,1]$ and $s \ge 1$, since $\xi_s$
is an indicator function. For definiteness, take $p=1$.

For $\xi_s$ as in \eqref{eq:xi},
by \eqref{eq:Prob1} the inequality \eqref{eq:Rs} turns into
an equality with $r_{s}(x,y):=s|x|$ if $y\preceq x$ and
$r_{s}(x,y):=\infty$ if $y$ is not dominated by $x$.

Throughout the section, for a function $f:[1,\infty) \to \R_+$, we will
write $f(s)=\mathcal{O}(\log^{d-1} s)$ to mean that
$f(s)/\log^{d-1} s$ is uniformly bounded for all $s \geq 1$. It is
well known (see, e.g., \cite{bai:dev:hwan:05}) that for all
$\alpha >0$,
\begin{equation}
	\label{eq:mean}
	s\int_{\XX} e^{-\alpha s |x|}\diff x=\mathcal{O}(\log^{d-1} s).
\end{equation}
In particular, by the Mecke formula,
$\E F_s = s\int_{\XX} e^{-s |x|}\diff x=\mathcal{O}(\log^{d-1}
s)$. Further, by the multivariate Mecke formula (see, e.g.,
\cite[Th.~4.4]{last:pen}),
\begin{equation*}
\Var(F_s) = \E F_s - (\E F_s)^2
+ s^2 \iint_{D} \Prob{x \text{ and } y
	\text{ are both minimal points in }
	\sP_s+\delta_x+\delta_y} \diff x \diff y,
\end{equation*}
where $D$ is the set of $(x,y) \in \XX^2$ such that $x$ and $y$ are
incomparable, i.e., $x \not \succ y$ and $y \not \succ x$.  Hence,
following the proof of Theorem~1 in \cite{bai:dev:hwan:05}, there
exist finite positive constants $C_1$ and $C_2$ such that
\begin{equation}
\label{eq:var}
C_1 \log^{d-1} s \le \operatorname{Var}(F_s) \le C_2 \log^{d-1} s,
\quad s \ge 1.
\end{equation}
For $\alpha>0$, $s>0$, and $d\in\N$, define the function
$c_{\alpha,s}: \XX \to \R_+$ as
\begin{equation}
\label{eq:cal}
c_{\alpha,s} (y):=s\int_{\XX} \mathds{1}_{x\succ y}
e^{-\alpha s |x|} \diff x.
\end{equation}
In view of the Mecke formula and the Poisson empty space formula,
$c_{1,s} (y)$ is the expected number of minimal points in
$\sP_s$ that dominate $y \in \XX$. Also note that $g_{s}(y)$ and $h_s(y)$ from
\eqref{eq:g} is equal to $c_{\zeta,s}(y)$ with $\zeta=p/(40+10p)=1/50$, so that $G_s(y) \le 3 + 2 c_{\zeta,s}(y)^5$.

Next, we specify the function $q_{s}$ from \eqref{eq:g2s}.  By
\eqref{eq:Prob2}, we have
\begin{displaymath}
q_{s}(x_1,x_2)
= s \int_\XX \mathds{1}_{z\succ (x_1\vee x_2)} e^{-s|z|}\diff z
=c_{1,s}(x_1\vee x_2).
\end{displaymath}

Studying the function $c_{\alpha,s}$ is essential to understand the
behaviour of minimal points.
Note that $c_{\alpha,s}$ satisfies the scaling property
\begin{equation}
\label{eq:scale}
c_{\alpha,s}(y)=\alpha^{-1} c_{1,\alpha s}(y), \quad \alpha>0,\; s>0.
\end{equation}
This will often enable us to take $\alpha=1$ without loss of 
generality. The following lemma demonstrates the asymptotic behaviour
of the function $c_{\alpha,s}$ for large $s$. Before we state the
result, notice that for $i \in \N \cup \{0\}$ and $\alpha>0$,
$$
\int_{0}^\infty |\log w|^{i} e^{-\alpha w}\diff w \le \int_0^1 |\log w|^i \diff w + \int_1^\infty w^i e^{-\alpha w}\diff w \le  \int_0^1 |\log w|^i \diff w + \frac{\Gamma (i+1)}{\alpha^{i+1}}.
$$
Since any positive integer power of logarithm is integrable near zero,
for all $i \in \N \cup \{0\}$ and $\alpha>0$,
\begin{equation}
\label{eq:Gamma}
\int_{0}^\infty |\log w|^{i} e^{-\alpha w}\diff w <\infty. 
\end{equation}

\begin{lemma}
\label{lemma:c-bound}
For all $\alpha>0$ and $s>0$, 
\begin{displaymath}
	c_{\alpha,s}(y)\leq \frac{D}{\alpha}
	e^{-\alpha s|y|/2}\Big[1+\big|\log(\alpha s|y|)\big|^{d-1}\Big], \quad y \in \XX
\end{displaymath}
for a constant $D$ that depends only on the dimension $d \in \N$.
\end{lemma}
\begin{proof}
The result is trivial when $d=1$, so we assume $d \ge 2$. By
\eqref{eq:scale}, we can also assume that $\alpha=1$. The following
derivation is motivated by those used to calculate the mean of
the number of minimal points in
\cite[Sec.~2]{bai:dev:hwan:05}. Changing variables $u=s^{1/d}x$ in
the definition of $c_{1,s}$ to obtain the first equality, and
letting $z^{(i)}=-\log u^{(i)}$, $i=1,\dots,d$, in the second, for
$y \in \XX$, we obtain
\begin{align*}
	c_{1,s}(y) 
	&=\int_{\times_{i=1}^d [s^{1/d}y^{(i)}, s^{1/d}]} e^{-|u|} \diff u\\
	&=\int_{\times_{i=1}^d \big[-d^{-1}\log s, -d^{-1}\log s -
		\log y^{(i)}\big]}
	\exp \bigg\{-e^{-\sum_{j=1}^d z^{(j)}} - \sum_{j=1}^d z^{(j)}
	\bigg\} \diff z.
\end{align*}
Next, we change variables by letting $v=(v^{(1)}, \dots, v^{(d)})$
with $v^{(i)}:=z^{(i)}+\cdots+z^{(d)}$, $i=1,\dots,d$. Note that the
integrand is only a function of $v^{(1)}$. Taking into account the
integration bounds on $z^{(i)}$, we have 
\begin{displaymath}
	v^{(1)} - \bigg(- \frac{i-1}{d} \log s - \sum_{j=1}^{i-1} \log y^{(i)}\bigg)
	\le v^{(i)} \le - \frac{d-i+1}{d} \log s - \sum_{j=i}^d \log
	y^{(i)},
	\quad 2 \le i \le d.
\end{displaymath}
Thus, for each $2 \le i \le d$, the integration variable $v^{(i)}$
belongs to an interval of length at most $(-\log (s|y|) -
v^{(1)})$. Using the substitution $w=e^{-v^{(1)}}$ in the second
step and Jensen's inequality in the last one, we obtain
\begin{align*}
	c_{1,s}(y)& \le \int_{-\log s}^{-\log (s|y|)} \Big(-\log (s|y|) -
	v^{(1)}\Big)^{d-1} \exp\Big\{-e^{-v^{(1)}} - v^{(1)}\Big\} \diff v^{(1)}\\
	& =\int_{s|y|}^{s} \Big(\log w-\log (s|y|)\Big)^{d-1}e^{-w}
	\diff w \\
	& \le 2^{d-2} e^{-s|y|/2} \bigg[\big|\log (s|y|)\big|^{d-1}
	+ \int_{s|y|}^s |\log w|^{d-1} e^{-w/2}\diff w \bigg].
\end{align*}
The result now follows by \eqref{eq:Gamma}.
\end{proof}

Before proceeding to estimate the bound in Theorem~\ref{thm:KolBd}, we
need some estimates of integrals involving $c_{\alpha,s}$ and
$|x|$. We will often use the following representation: for $\alpha>0$,
$s \ge 1$ and $i \in \N$,
\begin{align}\label{eq:crep}
s\int_\XX c_{\alpha,s}(x)^i \diff x
&=s\int_\XX \prod_{j=1}^i \Big(s\int_\XX \mathds{1}_{z_j \succ x}
e^{- \alpha s \sum_{j=1}^i |z_j| }\diff z_j\Big) \diff x\nonumber \\
& =s^{i+1} \int_{\XX^i} \big|z_1 \wedge \dots\wedge
z_i\big| e^{- \alpha s \sum_{j=1}^i |z_j| } \diff(z_1, \dots, z_i).
\end{align}

\begin{lemma}\label{lem:intbd}
For all $i \in \N$ and $\alpha>0$,
\begin{gather}
	\label{eq:c1}
	s\int_{\XX} c_{\alpha,s}(y)^i \diff y=
	\mathcal{O}(\log^{d-1} s),\\ 
	\label{eq:c2}
	s\int_{\XX}  \left(s \int_{\XX}e^{-\alpha s|x\vee
		y|}\diff x\right)^i \diff y=
	\mathcal{O}(\log^{d-1} s),\\ 
	\label{eq:c3}
	s \int_{\XX} \left(s\int_{\XX} c_{\alpha,s}(x\vee y) \diff x
	\right)^i \diff y
	= \mathcal{O}(\log^{d-1} s), 
\end{gather}
where the constants in the bounds on the right-hand sides may
depend on $i$.
\end{lemma}
\begin{proof}
As in Lemma~\ref{lemma:c-bound}, without loss of generality let
$\alpha=1$ and $s \ge 1$.  We first prove
\eqref{eq:c1}. 
For $i \in \N$, by Lemma~\ref{lemma:c-bound} and Jensen's inequality, we have
\begin{equation}
	\label{eq:aux}
	s\int_{\XX} c_{1,s}(y)^i \diff y
	\le 2^{i-1} D^i \left[s \int_{\XX} e^{- is|y|/2} \diff y
	+  s \int_{\XX} e^{- is|y|/2} \big|\log(s|y|)\big|^{i(d-1)}
	\diff y\right],
\end{equation}
with $D$ as in Lemma~\ref{lemma:c-bound}. The first
summand is of the order of $\log^{d-1} s$ by \eqref{eq:mean}. For
the second summand, we employ a similar substitution as in
Lemma~\ref{lemma:c-bound} and \cite{bai:dev:hwan:05}:
\begin{align*}
	& s \int_{\XX}  e^{- is|y|/2} \big|\log( s|y|)\big|^{i(d-1)} \diff y
	\le \int_{[0,s^{1/d}]^d} e^{-|u|/2} \big|\log |u|\big|^{i(d-1)}
	\diff u \qquad\;\;\;  (u=s^{1/d} x)\\
	& =\int_{[-d^{-1}\log s, \infty)^d} \exp
	\left\{- e^{-\frac{1}{2}\sum_{j=1}^d z^{(j)}} -
	\sum_{j=1}^d z^{(j)} \right\}
	\Bigg|\sum_{j=1}^d z^{(j)}\Bigg|^{i(d-1)} \diff z
	\quad\;\;\,\; (z^{(j)}=-\log u^{(j)})\\
	&\leq 
	\int_{-\log s}^\infty
	\big(\log s + v^{(1)}\big)^{d-1}
	\exp\Big\{-e^{-v^{(1)}/2} - v^{(1)}\Big\}
	|v^{(1)}|^{i(d-1)} \diff v^{(1)}
	\quad \quad \; (v^{(i)}=\sum_{j=i}^d z^{(j)})\\
	&= 
	\int_{0}^s \big(\log s -\log
	w\big)^{d-1} e^{-\sqrt{w}} |\log w|^{i(d-1)} \diff w \qquad\qquad\qquad\qquad\qquad \quad \, (w=e^{-v^{(1)}})\\
	&\le 2^{d-2} 
	\left[\log^{d-1} s
	\int_{0}^\infty e^{-\sqrt{w}} |\log w|^{i(d-1)}
	\diff w
	+ \int_{0}^\infty e^{-\sqrt{w}} |\log
	w|^{(i+1)(d-1)} \diff w\right],
\end{align*}
where the last step is due to Jensen's inequality. Finally, by
substituting $t=\sqrt{w}$ and using that $t e^{-t/2} \le 2$ for
$t \ge 0$, we have 
\begin{displaymath}
	\int_{0}^\infty e^{-\sqrt{w}} |\log w|^{j} \diff w
	=\int_{0}^\infty 2^{1+j} t e^{-t} |\log t|^{j} \diff t
	\le 2^{2+j} \int_{0}^\infty  e^{-t/2} |\log t|^{j} \diff t,
	\quad j\in\N.
\end{displaymath}
The result now follows by \eqref{eq:Gamma}.

Next, we move on to proving \eqref{eq:c2}. For
$x \in \XX$ and $I \subseteq \{1,\dots,d\}$, we write $x^I$ for
the subvector $(x^{(i)})_{i \in I}$. Assume that $x\vee y=(x^I,y^J)$ with
$J:=I^c$. 
Note that by Jensen's inequality, we have
\begin{equation}\label{eq:maxsplit}
	s\int_{\XX}  \left(s \int_{\XX}e^{- s|x\vee
		y|}\diff x\right)^i \diff y\le 2^{(i-1)d}
	\sum_{I \subseteq \{1,\dots,d\}} s\int_{\XX}
	\left(s \int_\XX \mathds{1}_{x^I \succ y^I, x^J \prec y^J}e^{- s|x^I| |y^J| }\diff x\right)^i \diff y.
\end{equation}
First, if $I=\emptyset$, splitting the exponential
into the product of two exponentials with the power halved, using $t^i e^{-t}
\le i!$ for $t\geq0$, and referring to \eqref{eq:mean} yield that 
\begin{displaymath}
	s\int_{\XX}  \left(s \int_\XX
	\mathds{1}_{x \prec y}e^{-s |y|} \diff x\right)^i \diff y 
	=	s\int_{\XX} (s|y|)^i  e^{-i s |y|} \diff y
	=\mathcal{O}(\log^{d-1} s).
\end{displaymath}
Next, assume that $I$ is nonempty and of cardinality $m$, with
$1 \le m \le d$. As a convention, let $|y^\emptyset|:=1$ for all
$y \in \XX$. Using Lemma~\ref{lemma:c-bound} with $\alpha=1$
and Jensen's inequality in the second step, we obtain
\begin{align*}
	s\int_{\XX}   \left(s \int_\XX \mathds{1}_{x^I \succ y^I, x^J \prec y^J}
	e^{-s |x^I|\,|y^J|} \diff x\right)^i  \diff y 
	&=	s\int_{\XX}\left(s |y^J| \int_{[0,1]^m} \mathds{1}_{x^I \succ y^I}   e^{-s |x^I|\,|y^J|}
	\diff x^I\right)^i \diff y\\
	& \le D^{i} 2^{i-1} s \int_{\XX} e^{-i
		s|y|/2}\Big[1+\big|\log(s|y|)\big|^{i(m-1)}\Big]\diff y,
\end{align*}
with $D$ as in Lemma~\ref{lemma:c-bound}. The two summands can be
bounded in the same manner as it was done for \eqref{eq:aux},
providing a bound of the order of $\log^{d-1} s$. The bound in
\eqref{eq:c2} now follows from \eqref{eq:maxsplit}.

Finally, we confirm \eqref{eq:c3}. Using that $te^{-t} \le 1$ for
$t \ge 0$ in the first inequality, we have
\begin{align*}
	s &\int_{\XX} \left(s\int_{\XX} c_{1,s}(x\vee y) \diff x \right)^i
	\diff y\\
	&= s^{2i+1} \int_{\XX} \int_{\XX^i} \bigg[\prod_{j=1}^{i} \int_\XX
	\mathds{1}_{z_j\succ x_j\vee y} 
	e^{- s \sum_{j=1}^i |z_j|} \diff z_j\bigg]
	\diff (x_1, \dots, x_i) \diff y\\
	&= s^{i+1} \int_{\XX^i} \bigg(s^i \prod_{j=1}^i |z_j|
	e^{- s \sum_{j=1}^i |z_j|/2 }\bigg)\; \big|z_1 \wedge \dots\wedge
	z_i\big| e^{- s \sum_{j=1}^i |z_j|/2 } \diff(z_1, \dots, z_i)\\
	&\le 2^i s^{i+1} \int_{\XX^i} \big|z_1 \wedge \dots\wedge
	z_i\big| e^{- s \sum_{j=1}^i |z_j|/2 } \diff(z_1, \dots, z_i)\\
	&\le 2^i s \int_\XX c_{1/2,s}(x)^i \diff x=\mathcal{O}(\log^{d-1} s),
\end{align*}
where we have also used \eqref{eq:crep} in the penultimate step and
\eqref{eq:c1} for the final step.
\end{proof}

Now we are ready to derive the bound in
Theorem~\ref{thm:KolBd}. Recall from
Section~\ref{sec:notat-main-results} the constants $\beta=p /(32+4 p)$
and $\zeta=p/(40+10p)$, which, in particular, satisfy that
$\zeta < 2\beta$. For our example, it suffices to let
$p=1$. Nonetheless, the following bounds are derived for any $\beta$
and $\zeta$, satisfying the above condition.

\begin{lemma}
\label{lem:intg1s}
For all $\beta\in(0,1/2)$, $\zeta\in(0,2\beta)$ and $f_{2\beta}$
defined at \eqref{eq:fa},
\begin{align*}
	s \int_{\XX} f_{2\beta}(x_1) \diff x_1=\mathcal{O}(\log^{d-1}s).
\end{align*}
\end{lemma}
\begin{proof} 
We first bound the integral of $f_{2\beta}^{(1)}$ defined at \eqref{eq:fal}. By
\eqref{eq:c1},
\begin{displaymath}
	s\int_{\XX} s\int_{\XX} e^{-2\beta r_s(x_2,x_1)}\diff x_2
	\diff x_1 =  s \int_{\XX} s \int_\XX \mathds{1}_{x_2 \succ x_1}
	e^{- 2\beta s |x_2|}\diff x_2
	\diff x_1
	=\mathcal{O}(\log^{d-1} s). 
\end{displaymath}
If $x_2 \succ x_1$, then
$c_{\zeta,s}(x_2) \le c_{\zeta,s}(x_1)$. Since $\zeta <2\beta$, by
\eqref{eq:c1},
\begin{align}\label{eq:2.1}
	s \int_{\XX} s \int_{\XX} c_{\zeta,s}(x_2)^5 e^{- 2\beta r_s(x_2,x_1)}  \diff x_2 \diff x_1
	&\le s \int_{\XX} c_{\zeta,s}(x_1)^5 s \int_\XX \mathds{1}_{x_2 \succ x_1} 
	e^{- 2\beta s |x_2|}  \diff x_2 \diff x_1\nonumber \\
	& \le s  \int_{\XX} c_{\zeta,s}(x_1)^6 \diff x_1=\mathcal{O}(\log^{d-1} s).
\end{align}
Since $G_s(y) \le 3 + 2 c_{\zeta,s}(y)^5$, combining the above two bounds, we obtain
\begin{displaymath}
	s \int_{\XX} f_{2\beta}^{(1)}(x_1) \diff x_1 =\mathcal{O}(\log^{d-1} s).
\end{displaymath}
We move on to $f_{2\beta}^{(2)}$. Using again that $te^{-t} \le 1$
for $t\geq0$ and \eqref{eq:mean}, we have
\begin{multline*}
	s\int_{\XX} s\int_{\XX} e^{-2\beta r_s(x_1,x_2)}\diff x_2
	\diff x_1 = s \int_{\XX} s \int_\XX
	\mathds{1}_{x_2 \prec x_1} e^{- 2\beta s |x_1|}
	\diff x_2 \diff x_1\\
	= s \int_{\XX} s|x_1|e^{- 2\beta s |x_1|}\diff x_1
	\le  s\beta^{-1} \int_{\XX} e^{- \beta s |x_1|}\diff x_1=\mathcal{O}(\log^{d-1} s). 
\end{multline*}
Also, $\zeta< 2\beta$ and
\eqref{eq:c1} yield that
\begin{align*}
	s & \int_{\XX} s \int_{\XX}c_{\zeta,s}(x_2)^5 e^{- 2\beta r_s(x_1,x_2)} 
	\diff x_2 \diff x_1\\
	&\le s \int_{\XX} c_{\zeta,s}(x_2)^5 \left(s\int_\XX \mathds{1}_{x_1 \succ x_2}
	e^{-\zeta s |x_1|} \diff x_1\right)\diff x_2
	=s \int_{\XX} c_{\zeta,s}(x_2)^6\diff x_2 =\mathcal{O}(\log^{d-1} s).
\end{align*}
Thus,
\begin{displaymath}
	s \int_{\XX} f_{2\beta}^{(2)}(x_1) \diff x_1=\mathcal{O}(\log^{d-1} s).
\end{displaymath}
It remains to  bound the integral of $f_{2\beta}^{(3)}$. For $\alpha<1$ and
$x \in \XX$, we have
\begin{multline}\label{eq:cl1}
	c_{1,s}(x)^\alpha=e^{-\alpha s |x|} \left(s \int_\XX \mathds{1}_{z \succ x}
	e^{-s(|z|-|x|)} \diff z\right)^\alpha 
	\le e^{-\alpha s |x|}\left[1+ s \int_\XX \mathds{1}_{z \succ x}
	e^{-s(|z|-|x|)} \diff z\right]\\ 
	 \le e^{-\alpha s |x|}\left[1+ s \int_\XX \mathds{1}_{z \succ x}
	e^{-\alpha s(|z|-|x|)} 
	\diff z\right]=e^{-\alpha s |x|} + c_{\alpha,s}(x).
\end{multline}
Thus, noticing that $2\beta<1$ and using Lemma~\ref{lem:intbd},
\begin{multline*}
	s \int_{\XX} s \int_{\XX} q_s(x_1,x_2)^{2\beta}
	\diff x_2 \diff x_1 = s \int_{\XX} s \int_{\XX} c_{1,s}(x_1\vee x_2)^{2\beta} 
	\diff x_2 \diff x_1\\
	\le	s^2 \int_{\XX^2}  e^{-2\beta s |x_1\vee x_2|}
	\diff (x_1,x_2)
	+ s^2 \int_{\XX^2} c_{2\beta,s}(x_1\vee x_2)
	\diff (x_1,x_2) =\mathcal{O}(\log^{d-1} s).
\end{multline*}
Finally, using \eqref{eq:cl1} and that $\zeta< 2\beta$ for the inequality, write
\begin{align*}
	s& \int_{\XX} s \int_{\XX} c_{\zeta,s}(x_2)^5 q_s(x_1,x_2)^{2\beta}
	\diff x_2 \diff x_1= s \int_{\XX}  s \int_{\XX}c_{\zeta,s}(x_2)^5
	c_{1,s}(x_1\vee x_2)^{ 2\beta} \diff x_2 \diff x_1\\
	&\le s^8 \int_{\XX} \int_{\XX} \int_{\XX^5} \mathds{1}_{z_1,\dots,z_5 \succ x_2}
	\int_\XX \mathds{1}_{z_6 \succ x_1\vee x_2}
	\exp\left\{-\zeta s \sum_{i=1}^6 |z_i|\right\} 
	\diff z_6 \diff(z_1,\dots,z_5) \diff x_2 \diff x_1\\
	&\qquad\qquad + s \int_{\XX}  s \int_{\XX}c_{\zeta,s}(x_2)^5 e^{-2\beta s
		|x_1\vee x_2|} \diff x_2 \diff x_1:=A_1 + A_2.
\end{align*}
By \eqref{eq:crep} and \eqref{eq:c1},
\begin{align*}
	A_1 &= s^8\int_{\XX^6}|z_6|\; \big|z_1 \wedge\dots\wedge z_6\big|
	\exp\left\{-\zeta s \sum_{i=1}^6 |z_i|\right\} \diff (z_1,\dots,z_6)\\
	&\le 2 s^7 \int_{\XX^6} |z_1 \wedge\dots\wedge z_6|
	\exp\left\{-\zeta s \sum_{i=1}^6 |z_i|/2\right\}\diff (z_1,\dots,z_6)\\
	&=2 s \int_\XX c_{\zeta/2,s}(x)^6 \diff x=\mathcal{O}(\log^{d-1} s).
\end{align*}
Furthermore, by the Cauchy-Schwarz inequality and Lemma~\ref{lem:intbd},
\begin{displaymath}
	A_2 \le \left(s \int_\XX c_{\zeta,s} (x_2)^{10} \diff x_2\right)^{1/2}
	\left(s \int_\XX \left(s \int_\XX e^{-2\beta s |x_1\vee x_2|} 
	\diff x_1\right)^2 \diff x_2\right)^{1/2} =\mathcal{O}(\log^{d-1} s).
\end{displaymath}
Therefore,
\begin{displaymath}
	s \int_{\XX} f_{2\beta}^{(3)}(x_1) \diff x_1 =\mathcal{O}(\log^{d-1} s),
\end{displaymath}
concluding the proof.
\end{proof}

\begin{lemma}\label{lem:intbd1}
For $\alpha_1,\alpha_2>0$,
\begin{displaymath}
	s \int_{\XX}  \left(s \int_{\XX}c_{\alpha_1,s}(x)^5 e^{-\alpha_2 s
		|x\vee y|}  \diff x\right)^2 \diff y 
	=\mathcal{O}(\log^{d-1} s).
\end{displaymath}
\end{lemma}
\begin{proof}
Since $c_{\alpha,s}$ is decreasing in $\alpha$ and in view of
\eqref{eq:scale}, it suffices to prove the result with both $\alpha_1$ and
$\alpha_2$ replaced by $1$.  We split the inner integral into
integration domains corresponding to the cases when
$x\vee y=(x^I,y^J)$ with $J=I^c$ for $I \subseteq \{1,\dots,d\}$. First, if $I=\{1,\dots,d\}$, then
using monotonicity of $c_{1,s}$ and \eqref{eq:c1}, we have
\begin{multline*}
	s \int_{\XX}  \left(s \int_\XX \mathds{1}_{x \succ
		y}c_{1,s}(x)^5 e^{-s |x\vee y|} \diff x\right)^2 \diff y\\
	\le s \int_{\XX} c_{1,s}(y)^{10}  \left(s
	\int_\XX \mathds{1}_{x \succ y}e^{- s |x|}\diff x\right)^2 \diff y 
	\le s \int_{\XX} c_{1,s}(y)^{12} \diff y =\mathcal{O}(\log^{d-1} s).
\end{multline*}
By writing the function $|\cdot|$ as the product of coordinates and passing
to the one-dimensional case, it is easy to see that for $a,b,y \in \XX$,
\begin{equation}
	\label{eq:l-inequality}
	|a \wedge y|\;|b\wedge y|
	\le |a \wedge b \wedge y|\; |y|.
\end{equation}
Hence, when $I=\emptyset$, 
\begin{align*}
	&s \int_{\XX} \left(s \int_{\XX} \mathds{1}_{x \prec y}c_{1,s}(x)^5 e^{-s |x\vee y|} 
	\diff x\right)^2 \diff y
	= s \int_{\XX} e^{-2 s |y|}\left(s \int_\XX \mathds{1}_{x \prec
		y}c_{1,s}(x)^5  \diff x\right)^2 \diff y\\
	& \le s^{13} \int_\XX \int_{\XX^2} \mathds{1}_{x_1,x_2 \prec y}
	\int_{\XX^{10}} \mathds{1}_{z_1,\dots, z_5 \succ x_1}
	\mathds{1}_{z_6,\dots,z_{10} \succ x_2} e^{- s|y| - s \sum_{i=1}^{10}  |z_i|} 
	\diff (z_1,\dots,z_{10}) \diff(x_1,x_2) \diff y\\
	&=s^{13} \int_\XX \int_{\XX^{10}}
	\big|z_1\wedge\dots\wedge z_5\wedge y\big|\;
	\big|z_6 \wedge\dots\wedge z_{10} \wedge y\big|
	e^{- s|y|- s \sum_{i=1}^{10} |z_i|} 
	\diff (z_1,\dots,z_{10}) \diff y\\
	&\le s^{13} \int_\XX \int_{\XX^{10}}
	\big|z_1\wedge\dots\wedge z_{10}\wedge y\big|\;
	|y|\; e^{ - s |y|- s \sum_{i=1}^{10} |z_i|}
	\diff (z_1,\dots,z_{10}) \diff y,
\end{align*}
where in the final step, we have used \eqref{eq:l-inequality} with
$a:=z_1 \wedge\dots\wedge z_5$ and
$b:=z_6 \wedge\dots\wedge z_{10}$.
Splitting the exponential into product of two exponentials with
powers halved, and using the fact that
\begin{displaymath}
	s|y| e^{ - s |y|/2- s \sum_{i=1}^{10} |z_i|/2} \leq 2,
\end{displaymath}
we obtain by \eqref{eq:c1} that the last integral is bounded by
\begin{align*}
	&2 s^{12} \int_\XX \int_{\XX^{10}} \big|z_1\wedge\dots\wedge z_{10} \wedge y\big|
	\; e^{- s |y|/2 - s \sum_{i=1}^{10} |z_i|/2 }
	\diff (z_1,\dots,z_{10}) \diff y\\
	&=2 s \int_\XX \left(s\int_\XX \mathds{1}_{y\succeq x}
	e^{-s|y|/2}\diff y\right)
	\prod_{i=1}^{10} \left(s\int_\XX \mathds{1}_{z_i\succeq x}
	e^{- s|z_i|/2}\diff z_i\right) \diff x\\
	&=2 s \int_\XX c_{1/2,s}(x)^{11}\diff x
	=\mathcal{O}(\log^{d-1} s).
\end{align*}

Next, assume that $d \ge 2$ and $I$ is nonempty of cardinality $m$ with
$1 \le m \le d-1$. Using monotonicity of $c_{1,s}$ in the first step
and Lemma~\ref{lemma:c-bound} in the last step upon identifying the
integral as the function given by \eqref{eq:cal} in the space of
dimension $m$, we have
\begin{align}
	\label{eq:incom}
	s& \int_{\XX}  \left(s \int_\XX \mathds{1}_{x^I \succ y^I, x^J \prec y^J}
	c_{1,s}(x)^5 e^{- s |x^I|\,|y^J|} \diff x\right)^2 \diff y\nonumber\\
	& \le s \int_{\XX}  \left(s \int_\XX \mathds{1}_{x^I \succ y^I, x^J \prec y^J}
	c_{1,s}(x^J,y^I)^5 e^{- s |x^I|\,|y^J|} \diff x\right)^2 \diff y\nonumber\\
	& = s \int_{\XX}  \left( \int_{[0,1]^m} \mathds{1}_{x^I \succ y^I} e^{- s
		|x^I|\,|y^J|} \diff x^I\right)^2 \left(s \int_{[0,1]^{d-m}} \mathds{1}_{x^J \prec y^J}
	c_{1,s}(x^J,y^I)^5 \diff x^J\right)^2 \diff y\nonumber\\
	& \le D^2 s \int_{\XX} \frac{e^{- s|y|}}{s^2 |y^J|^2}
	\left(1+ \big|\log ( s|y|)\big|^{2(m-1)}\right)
	\left(s \int_{[0,1]^{d-m}} \mathds{1}_{x^J \prec y^J} c_{1,s}(x^J,y^I)^5 \diff x^J\right)^2 \diff y,
\end{align}
with $D$ as in Lemma~\ref{lemma:c-bound}. We will now estimate the
integral inside \eqref{eq:incom}. Using Lemma~\ref{lemma:c-bound}
and Jensen's inequality in the first step, substituting
$u=(s|y^I|)^{1/(d-m)} x^J$ in the second step,
letting $z^{(i)}=\log u^{(i)}$, $i=1,\dots,d-m$, in the third one,
$v^{(1)}=\sum_{i=1}^{d-m} z^{(i)}$ in the fourth, $w=e^{-v^{(1)}}$
in the fifth, and, finally, Jensen's inequality in the penultimate
step, we obtain that
\begin{align*}
	s&|y^I| \int_{[0,1]^{d-m}} \mathds{1}_{x^J \prec y^J}  c_{1,s}(x^J,y^I)^5 \diff x^J\\
	&\le 16 D^5 s|y^I| \int_{[0,1]^{d-m}} \mathds{1}_{x^J \prec y^J} e^{-5 s |x^J|\,|y^I|/2}
	\Big(1+ \big|\log (s|x^J|\,|y^I|)\big|^{5(d-1)}\Big) \diff x^J\\
	&= 16 D^5 \int_{\big[0,(s|y^I|)^{\frac{1}{d-m}}\big]^{d-m}} \mathds{1}_{u \prec (s|y^I|)^{\frac{1}{d-m}}y^J} e^{-\frac{5}{2}|u|}
	\Big(1+ \big|\log (|u|)\big|^{5(d-1)}\Big) \diff u\\
	&= 16 D^5 \int_{\times_{j \in J}\big[-\frac{\log (s|y^I|)}{d-m} - \log y^{(j)},\infty\big)}
	\exp \left\{-e^{-\frac{5}{2}\sum_{i=1}^{d-m} z^{(i)}} -\sum_{i=1}^{d-m} z^{(i)}\right\}\\
	& \qquad \qquad\qquad\qquad\qquad\qquad\qquad \qquad\qquad\times \Big(1+ \Big|\sum_{i=1}^{d-m} z^{(i)}\Big |^{5(d-1)}\Big) \diff z\\
	&\leq  16 D^5 
	\int_{-\log s|y|}^\infty
	\Big(v^{(1)} + \log (s|y|)\Big)^{d-m-1}
	\exp\left\{-e^{-\frac{5}{2}v^{(1)}} - v^{(1)}\right\}
	\Big(1+|v^{(1)}|^{5(d-1)}\Big) \diff v^{(1)}\\
	&= 16 D^5 
	\int_{0}^{s|y|}e^{-w^{5/2}}
	\Big(\log (s|y|) - \log w\Big)^{d-m-1}
	\Big(1+|\log w|^{5(d-1)}\Big) \diff w\\
	& \le 16 D^5 2^{d-m-2} 
	\bigg[|\log (s|y|)|^{d-m-1}
	\int_{0}^{s|y|} \Big(1+|\log w|^{5(d-1)}\Big) \diff w\\
	& \qquad\qquad \qquad \qquad \qquad\qquad \qquad \qquad \qquad
	+ \int_{0}^{s|y|}|\log w|^{d-m-1}\Big(1+|\log w|^{5(d-1)}\Big) \diff w \bigg]\\
	& \le D' s|y| \bigg[1+ \sum_{i=1}^{6(d-1)-m} \big|\log (s|y|)\big|^{i} \bigg]
\end{align*}
for a constant $D'$ depending only on $d$ and $m$, so that the bound
on the last integral in \eqref{eq:incom} is obtained by dividing by
$|y^I|$ on both sides. The last step relies on an elementary
inequality, saying that, for $l \in \N \cup \{0\}$ and
$a>0$, there exists a constant $b_l>0$ depending only on $l$ such
that
\begin{displaymath}
	\int_0^{a} |\log w|^l \diff w \le b_l a
	\left[1+\sum_{i=1}^{l} |\log a|^i\right].
\end{displaymath}
Plugging this in \eqref{eq:incom} and using Jensen's inequality, we obtain
\begin{align*}
	s & \int_{\XX}  \left(s \int_\XX \mathds{1}_{x^I \succ y^I, x^J \prec
		y^J}c_{1,s}(x)^5 e^{- s |x^I|\,|y^J|} \diff x\right)^2 \diff y\\
	& \le D'' s \int_{\XX} \frac{e^{-s|y|}}{s^2 |y^J|^2}
	\left(1+ |\log (s|y|)|^{2(m-1)}\right)
	s^2 |y^J|^2 \left(1+ |\log (s|y|)|^{12(d-1)-2m}\right) \diff y\\
	&=\mathcal{O}(\log^{d-1} s)
\end{align*}
for some constant $D''$ depending on $d$ and $m$, where the last
step is argued similarly as for \eqref{eq:aux}. Summing
over all possible $I \subseteq\{1,\dots,d\}$ yields the desired
conclusion.
\end{proof}


\begin{lemma}\label{lem:intg2s}
For $\beta\in(0,1/2)$, $\zeta\in(0,\beta)$ and $f_{\beta}$ defined at \eqref{eq:fa},
\begin{align*}
	s \int_{\XX} f_{\beta}(x_1)^2 \diff x_1 =\mathcal{O}(\log^{d-1} s).
\end{align*}
\end{lemma}
\begin{proof}
As in Lemma~\ref{lem:intg1s}, we consider integrals of squares of
$f_\beta^{(i)}$ for $i=1,2,3$ separately. By \eqref{eq:c1},
\begin{displaymath}
	s\int_{\XX} \left(s\int_{\XX} e^{-\beta r_s(x_2,x_1)}\diff x_2\right)^2
	\diff x_1 =s \int_{\XX} \left(s \int_\XX \mathds{1}_{x_2 \succ
		x_1} 
	e^{- \beta s |x_2|} 
	\diff x_2\right)^2 \diff x_1 =\mathcal{O}(\log^{d-1} s). 
\end{displaymath}
Arguing as in \eqref{eq:2.1}, using monotonicity of
$c_{\zeta,s}$, $\zeta<\beta$,  and \eqref{eq:c1}, we have
\begin{align*}
	s \int_{\XX} \left(s \int_{\XX}c_{\zeta,s}(x_2)^5 e^{- \beta r_s(x_2,x_1)} 
	\diff x_2\right)^2 \diff x_1
	& \le s \int_{\XX} c_{\zeta,s} (x_1)^{10} \left(s \int_\XX \mathds{1}_{x_2 \succ x_1} 
	e^{- \beta s |x_2|}  \diff x_2\right)^2 \diff x_1\\
	& \le s  \int_{\XX} c_{\zeta,s} (x_1)^{12}\diff x_1 =\mathcal{O}(\log^{d-1} s).
\end{align*}
Recalling that $G_s(y) \le 3 + 2 c_{\zeta,s}(y)^5$, combining the above bounds and using Jensen's inequality yield
\begin{displaymath}
	s \int_{\XX} f_{\beta}^{(1)}(x_1)^2 \diff x_1 =\mathcal{O}(\log^{d-1} s).
\end{displaymath}

Next, we integrate the square of $f_{ \beta}^{(3)}$. Using
\eqref{eq:cl1} and Lemma~\ref{lem:intbd},
\begin{multline}\label{eq:p1}
	s \int_{\XX} \left(s \int_{\XX} q_s(x_1,x_2)^{\beta}
	\diff x_2\right)^2 \diff x_1
	= s \int_{\XX} \left(s \int_{\XX} c_{1,s}(x_1\vee x_2)^{ \beta}
	\diff x_2\right)^2 \diff x_1\\
	\le 2s \int_{\XX}\left(s \int_\XX  e^{-\beta s |x_1\vee x_2|} 
	\diff x_2 \right)^2 \diff x_1 
	+ 2s \int_{\XX} \left(s \int_\XX c_{\beta,s}(x_1\vee x_2) 
	\diff x_2\right)^2 \diff x_1=\mathcal{O}(\log^{d-1} s).
\end{multline}
Again using \eqref{eq:cl1},
\begin{align*}
	s &\int_{\XX} \left(s \int_{\XX} c_{\zeta,s}(x_2)^5 q_s(x_1,x_2)^{\beta}
	\diff x_2\right)^2 \diff x_1
	= s \int_{\XX} \left( s \int_{\XX}c_{\zeta,s}(x_2)^5
	c_{1,s}(x_1\vee x_2)^{\beta} \diff x_2\right)^2 \diff x_1\\
	&\le 2s \int_{\XX}  \left(s \int_{\XX}c_{\zeta,s}(x_2)^5 e^{-\beta s
		|x_1\vee x_2|} \diff x_2\right)^2 \diff x_1
	+ 2s \int_{\XX}  \left(s \int_{\XX}c_{\zeta,s}(x_2)^5
	c_{\beta,s}(x_1\vee x_2) \diff x_2\right)^2 \diff x_1\\
	&:=2(A_1 + A_2).
\end{align*}
By Lemma~\ref{lem:intbd1}, we have $A_1 =\mathcal{O}(\log^{d-1} s)$. For $x_1\in\XX$ and $(x_{21},x_{22})\in\XX^2$, denote 
\begin{multline*}
	A(x_1,x_{21}, x_{22}):=\Big\{(z_1,\dots,z_{12})\in\XX^{12}:\\
	z_1,\dots,z_5 \succ x_{21},\; z_6,\dots,z_{10}
	\succ x_{22},\; z_{11} \succ x_1\vee x_{21},\; z_{12} \succ x_1\vee
	x_{22}\Big\}.
\end{multline*}
By applying \eqref{eq:l-inequality} twice we have 
\begin{displaymath}
	|a \wedge x|\; |b\wedge y|\; |x\wedge y|
	\leq |a\wedge b\wedge x\wedge y|\; |x|\; |y|
	\leq |a\wedge b\wedge x\wedge y|\, (|x|+|y|)^2,
	\quad a,b,x,y\in\XX. 
\end{displaymath}
Using this with $a:=z_1 \wedge\dots\wedge z_5$,
$b:=z_6 \wedge\dots\wedge z_{10}$, $x:=z_{11}$, $y:=z_{12}$ in the
third step, \eqref{eq:crep} in the penultimate step, and
\eqref{eq:c1} in the last one, we obtain
\begin{align*}
	A_2 &\le s^{15} \int_{\XX} \int_{\XX^2} \int_{A(x_1,x_{21},x_{22})} 
	e^{-\zeta s \sum_{i=1}^{12} |z_i| }
	\diff (z_1,\dots,z_{12})
	\diff (x_{21}, x_{22}) \diff x_1\\
	&= s^{15} \int_{\XX^{12}} e^{-\zeta s \sum_{i=1}^{12} |z_i|}
	|z_1 \wedge\dots\wedge z_5 \wedge z_{11}|\; |z_6 \wedge\dots\wedge
	z_{10} \wedge z_{12}|\;
	|z_{11} \wedge z_{12}| \diff (z_1,\dots,z_{12})\\
	&\le s^{15} \int_{\XX^{12}} e^{-\zeta s \sum_{i=1}^{12} |z_i|}
	|z_1 \wedge\dots\wedge z_{12}|\; \big(|z_{11}| + |z_{12}|\big)^2  
	\diff (z_1,\dots,z_{12})\\
	& \le (8/\zeta^2) s^{13}  \int_{\XX^{12}} 
	e^{-\zeta s \sum_{i=1}^{12} |z_i|/2}|z_1\wedge\dots\wedge z_{12}| 
	\diff (z_1,\dots,z_{12})\\
	&= (8/\zeta^2) s\int_\XX c_{\zeta/2,s}(x)^{12} \diff x =\mathcal{O}(\log^{d-1} s),
\end{align*}
where for the last inequality we have used that 
\begin{displaymath}
	s^2 \big(|z_{11}|+|z_{12}|\big)^2
	e^{-\zeta s \sum_{i=1}^{12} |z_i|/2 } \le 8/\zeta^2.
\end{displaymath}
Combining the bounds on $A_1$ and $A_2$ with \eqref{eq:p1} yields that
\begin{displaymath}
	s \int_{\XX} f_{\beta}^{(3)}(x_1)^2 \diff x_1 =\mathcal{O}(\log^{d-1} s).
\end{displaymath}

For the integral of the square of $f_{\beta}^{(2)}$, arguing as in
Lemma~\ref{lem:intg1s} and using the inequality $t^2 e^{-t} \le 2$ for
$t\geq0$, we have
\begin{multline*}
	s\int_{\XX} \left(s\int_{\XX} e^{-\beta r_s(x_1,x_2)}\diff x_2 \right)^2
	\diff x_1
	= s \int_{\XX} \left(s \int_\XX \mathds{1}_{x_2 \prec x_1} e^{-\beta s |x_1|} 
	\diff x_2\right)^2 \diff x_1\\
	= s/\beta^2 \int_{\XX} (\beta s|x_1|)^2\; e^{-2\beta s |x_1|}\diff x_1
	\le 2 s/\beta^2 \int_{\XX} e^{-\beta s |x_1|}\diff x_1
	=\mathcal{O}(\log^{d-1} s). 
\end{multline*}
Changing order of integration in the second step, using the
Cauchy--Schwarz inequality in the third one, and referring to
\eqref{eq:c1} in the last step yield that
\begin{align*}
	s&\int_{\XX} \left(s\int_{\XX} c_{\zeta,s}(x_2)^5 e^{-\beta r_s(x_1,x_2)}\diff x_2 \right)^2
	\diff x_1
	=s \int_{\XX} \left(s \int_\XX \mathds{1}_{x_2 \prec x_1} c_{\zeta,s} (x_2)^5
	e^{-\beta s |x_1|} \diff x_2\right)^2 \diff x_1\\
	&= s^2 \int_{\XX^2} c_{\zeta,s}(x)^5 c_{\zeta,s}(y)^5 c_{\beta,s}(x
	\vee y) 
	\diff (x,y)\le \left(s\int_\XX c_{\zeta,s}(x)^{10}
	\diff x\right)^{1/2}A_2^{1/2} =\mathcal{O}(\log^{d-1} s),
\end{align*}
where $A_2$ is defined above.  Thus,
\begin{displaymath}
	s \int_{\XX} f_{\beta}^{(2)}(x_1)^2 \diff x_1 =\mathcal{O}(\log^{d-1} s).
\end{displaymath}
Combining, we obtain the desired result.	
\end{proof}

Since $2 \beta=2p/(32+4p)<1$, to
compute the bound, it suffices to provide a bound on the integral of
$(\kappa_s+g_{s})^{\beta} G_s$ for any $\beta \in (0,1)$.

\begin{lemma}\label{lem:qgint} 
For $\beta,\zeta\in(0,1)$, let $G_s$ and $\kappa_s$ be as in
\eqref{eq:g5} and \eqref{eq:p} respectively. Then
\begin{displaymath}
	s\int_\XX G_s(x)\big(\kappa_s(x)+g_{s}(x)\big)^\beta \diff x
	=\mathcal{O}(\log^{d-1} s).
\end{displaymath}
\end{lemma}
\begin{proof}
First note that
\begin{displaymath}
	\kappa_s(x)=\Prob{\xi_{s}(x, \sP_{s}+\delta_x)\neq 0}=e^{-s |x|}, \quad x \in \XX.
\end{displaymath}
Using the Cauchy--Schwarz inequality in the second step, by
\eqref{eq:mean} and \eqref{eq:c1},
\begin{multline*}
	s\int_\XX G_s(x)\kappa_s(x)^\beta \diff x \le 3s \int_\XX
	(1+c_{\zeta,s}(x)^5) e^{-\beta s |x|} \diff x\\
	\le 3s \int_\XX e^{-\beta  s |x|} \diff x 
	+ 3\left(s \int_\XX c_{\zeta,s}(x)^{10} \diff x\right)^{1/2}
	\left(s \int_\XX e^{-2\beta s |x|} \diff x\right)^{1/2}
	=\mathcal{O}(\log^{d-1} s).
\end{multline*}
Since $\beta \in (0,1)$, arguing as in \eqref{eq:cl1}, 
$$
c_{\zeta,s}(x)^\beta \le e^{-\beta \zeta |x|} + c_{\beta \zeta,s}(x).
$$
An application of \eqref{eq:mean} and \eqref{eq:c1} now yields
\begin{align*}
	s\int_\XX &G_s(x)g_{s}(x)^\beta \diff x 
	\le 3s \int_\XX \Big(1+c_{\zeta,s}(x)^5\Big) c_{\zeta,s}(x)^\beta\diff x\\
	&\le 3s \int_\XX e^{-\beta \zeta |x|} \diff x
	+ 3s \int_\XX c_{\beta \zeta,s}(x) \diff x
	+ 3s \int_\XX c_{\zeta,s}(x)^{5+\beta}\diff x=\mathcal{O}(\log^{d-1} s).
\end{align*}
Combining the above bounds, we obtain the desired conclusion.
\end{proof}

\begin{proof}[Proof of Theorem~\ref{thm:Pareto}.] 
By \eqref{eq:var}, $\Var(F_s) \ge C_1 \log^{d-1} s$ for all
$s \ge 1$. An application of Theorem~\ref{thm:KolBd} with
Lemmas~\ref{lem:intg1s}, \ref{lem:intg2s} and \ref{lem:qgint} now
yields the desired upper bound.

The proof of the optimality of the bound on the Kolmogorov distance follows by a general argument employed in the proof of \cite[Theorem~1.1, Eq.~(1.6)]{EG81}, which shows that the Kolmogorov distance between any integer-valued random variable, suitably normalized, and a standard normal random variable is always lower bounded by a constant times the inverse of the standard deviation, see Section~6 therein for further details. The variance upper bound in \eqref{eq:var} now yields the result.
\end{proof}

\section{Non-diffuse intensity measures and unbounded scores}
\label{sec:ex}

As discussed in the introduction, in addition to working with general
stabilization regions, our approach generalizes results in
\cite{LSY19} in two more ways. First, we allow for non-diffuse
intensity measures and, second, we can consider score functions that do
not have uniformly bounded moments over $x \in \XX$. In this section, we
demonstrate this with two examples. In Example~\ref{ex:1}, we consider
a Poisson process on the two dimensional integer lattice with the
counting measure as the intensity, which is non-diffuse. We derive a
quantitative central limit theorem for the number of isolated points
in this setup.

In Example~\ref{ex:2}, we consider isolated vertices in a random
geometric graph built on a stationary Poisson process on $\R^d$, where
two points are joined by an edge if the distance between them is at
most $\rho_s$ for some appropriate non-negative function $\rho_s$,
$s \ge 1$. Poisson convergence for the number of such isolated
vertices in different regimes has been extensively studied, see, e.g.,
\cite[Ch.~8]{pen03}. But, instead of considering the number of
isolated vertices, we consider the sum of values for a general
function evaluated at locations of isolated vertices, for instance,
the logarithms of scaled norms. As the logarithm is unbounded near the
origin, the score functions do not admit a uniform bound on their
moments. We note here that in both the examples below, it should be
possible to work with a binomial process as well, once a result
paralleling our Theorem~\ref{thm:Main} is proved in this setting. As
mentioned in the introduction, this can be done by following the
scheme in \cite{LSY19} suitably adapted to incorporate general
stabilization regions.

\begin{example}[Non-diffuse intensity]
\label{ex:1}
Let $\XX:=\ZZ^2$ and consider a Poisson process $\sP$ on $\ZZ^2$
with the intensity measure $\QQ$ being the counting measure on
$\ZZ^2$; so we let $s=1$ and omit it from the
subscripts. A point $x \in \sP$ is said to be isolated
in $\sP$ if all its nearest neighbors are unoccupied, i.e.,
$\sP(x+B)=0$, where $+$ denotes the Minkowski addition and
$B:=\{(0,\pm1), (\pm1,0)\}$, so that $x+B$ is the set comprising the
4 nearest neighbors of $x\in\ZZ^2$. Consider a weight function
$w: \ZZ^2 \to \R_+$, and for $i \in \N$ denote
$$W_i:=\sum_{x\in\ZZ^2} w(x)^i.$$ 
Assume that $W_1=\sum_{x\in\ZZ^2} w(x)<\infty$, which in particular implies that $w$ is bounded. Scaling $w$, assume
without loss of generality that $w$ is bounded by one. Consider the statistic $H\equiv H_1(\sP_1)$ defined at \eqref{eq:hs}
with
\begin{displaymath}
	\xi(x, \sP):=w(x) \mathds{1}_{\sP(x+B)=0}\;, \quad x\in\sP.
\end{displaymath}

For $x \in \ZZ^2$, defining the stabilization region
$R(x,\sP+\delta_x):=(x+B)$ if $x$ is isolated in $\sP+\delta_x$ and
$R(x,\sP+\delta_x):=\emptyset$ otherwise, we see that
\eqref{eq:ximon} and (A1) are trivially satisfied. Also, (A2) holds
with $p=1$ and $M_{1,1}(x)=w(x)$, while \eqref{eq:Rs} holds with
$r(x,y)=4$ for $x\in \ZZ^2$ and $y \in x+B$ and $r(x,y)=\infty$
otherwise. Next, notice that $\kappa(y)=e^{-4}$, $y \in \ZZ^2$, $\zeta=1/50$,
\begin{displaymath}
	g(y)=\sum_{x \in y+B, x\in\ZZ^2} e^{-4/50}=4 e^{-4/50}\;
	\text{ and } \; h(y)=\sum_{x \in y+B, x\in\ZZ^2} w(x)^{9/2}e^{-4/50},\qquad
	y \in \ZZ^2, 
\end{displaymath}
while $q(x_1,x_2) \le 4e^{-4}$ for $x_1,x_2 \in \ZZ^2$ with
$x_2-x_1\in B+B$ and $q(x_1,x_2)=0$ otherwise. Noticing that
$\max\{w(x)^2,w(x)^4, w(x)^{9/2}\}=w(x)^2$, we obtain that for all $\alpha>0$,
there exists a constant $C_\alpha$ such that
\begin{displaymath}
	f_\alpha(y) \le C_\alpha \sum_{x-y\in (B+B) \cup (B+B +B), x,y\in\ZZ^2} w(x)^2.
\end{displaymath}
Thus, with $\beta=1/36$, there exists a constant $C>0$ such that
\begin{displaymath}
	\QQ f_\beta^2 \le 
	C W_4,
	\quad
	\max\{\QQ f_{2\beta}, \QQ ((\kappa+g)^{2\beta}G) \}
	\le 
	C W_2.
\end{displaymath}
On the other hand, by the Mecke formula, we have
\begin{align*}
	\Var(H) &= \E \sum_{x \in \sP} w^2(x)
	\mathds{1}_{\sP(x+B)=0} - (\E H)^2\\
	&\qquad + \sum_{x \in \ZZ^2} \sum_{y \in (x+B)^c, y\in\ZZ^2}
	w(x)w(y)\Prob{(\sP+\delta_x+\delta_y)\big((x+B)\cup(y+B)\big)=0}\\
	& = e^{-4}W_2 - e^{-8} \sum_{x \in \ZZ^2}
	\sum_{y \in (x+B)} w(x)w(y) \\
	& \qquad+ \sum_{x \in \ZZ^2} \sum_{y \in (x+B)^c, y\in\ZZ^2}
	w(x)w(y)\Big(\Prob{\sP\big((x+B)\cup(y+B)\big)=0} - e^{-8}\Big)\\
	& \ge  e^{-4}W_2
	+ (e^{-7} - e^{-8})\sum_{x \in \ZZ^2}
	\sum_{y-x \in (B+B)} w(x)w(y)- e^{-8} \sum_{x \in \ZZ^2}
	\sum_{y \in (x+B)} w(x)w(y).
\end{align*}
Finally, noticing that
$$
\sum_{x \in \ZZ^2} \sum_{y \in (x+B)} w(x)w(y)
\le \sum_{x \in \ZZ^2} \sum_{y \in (x+B)} \frac{w(x)^2 + w(y)^2}{2} = 4 W_2,
$$
we obtain
$$
\Var(H)  \ge 
(e^{-4} - 4 e^{-8})W_2.
$$
Hence, an application of Theorem~\ref{thm:KolBd} yields that
\begin{align*}
	\max&\left\{d_{W}\left(\frac{H-\E H}{\sqrt{\Var H}},N\right) ,d_{K}\left(\frac{H-\E H}{\sqrt{\Var H}},N\right) \right\}\\
	&\qquad \leq \frac{C}{(W_2)^{1/2}}\left[1+\sqrt{\frac{W_4}{W_2}} + \frac{1}{W_2^{1/4}}\right] \le \frac{C}{(W_2)^{1/2}} \left[2+ \frac{1}{W_2^{1/4}}\right] ,
\end{align*}
for some constant $C>0$, where the final step is due to the observation that $W_4 \le W_2$. As an example, one can take
$w(x):=\mathds{1}_{x \in [-n,n]^2}$ for $n \in \N$ to see that the
distances on the left-hand side is bounded by $C/n$, which is
presumably optimal, since the variance is of the order $n^2$. In particular, arguing as in the proof of Theorem~\ref{thm:Pareto}, the bound on the Kolmogorov distance is of optimal order in this case.
\end{example}

\begin{example}[Weighted sum over isolated vertices in random
geometric graphs]\label{ex:2}
Let $\XX := \R^d$ with $d \ge 2$, and let $\sP_s$ be a Poisson
process on $\XX$ with intensity measure $s \QQ$ for $s \ge 1$ and
the Lebesgue measure $\QQ$. Fix $s \ge 1$. Given
$\rho_s>0$, consider a random geometric graph $G_s(\sP_s,\rho_s)$
with the vertex set $\sP_s$, where an edge joins two distinct
vertices $x$ and $y$ if $\|x-y\| \le \rho_s$, where $\|\cdot\|$
denotes the Euclidean norm. A vertex $x \in \sP_s$ is called
isolated if $\sP_s(B(x,\rho_s) \setminus \{x\} )=0$, where $B(x,\rho_s)$ denotes the
closed ball of radius $\rho_s$ centered at $x$. For a (possibly unbounded) weight
function $w_s: \R^d \to \R_+$ with
$\int_{\R^d} \max\{w_s(x),w_s(x)^8\}\diff x<\infty$, consider the
statistic $H_s$ defined at \eqref{eq:hs} with
\begin{displaymath}
	\xi_s(x,\sP_s):=w_s(x) \mathds{1}_{x\text{ is isolated in $\sP_s$}},
	\quad x \in \sP_s.
\end{displaymath}
For $x \in \XX$, letting $R_s(x,\sP_s+\delta_x):=B(x,\rho_s)$ if $x$
is isolated in $\sP_s+\delta_x$ and $\emptyset$ otherwise, we see
that \eqref{eq:ximon} and (A1) are satisfied. As in
Example~\ref{ex:1}, (A2) holds with $p=1$ and
$M_s=M_{s,1}(x):=w_s(x)$. Letting $r_s(x,y):=k_d s \rho_s^d$ for
$x\in \R^d$ and $y \in B(x,\rho_s)$, where $k_d$ is the volume of
the unit ball in $\R^d$, and $r_s(x,y):=\infty$ otherwise, one
verifies \eqref{eq:Rs}. Clearly, $\kappa_s(y)\le e^{-k_d s \rho_s^d}$ for $y \in \R^d$. Also, since $\zeta=1/50$, one has
\begin{displaymath}
	g_s(y) = k_d s\rho_s^d e^{-k_ds \rho_s^d /50}\;
	\text{ and } \; h_s(y)=s e^{-k_ds \rho_s^d /50} \int_{B(y,\rho_s)} w_s(x)^{9/2} \diff x, \quad y \in \R^d,
\end{displaymath}
while $q_s(x_1,x_2) \le k_d s \rho_s^d e^{-k_d s \rho_s^d}$ for
$x_1,x_2 \in \R^d$ with $\|x_2-x_1\|\le 2 \rho_s$ and
$q_s(x_1,x_2)=0$ otherwise. Next, we compute the variance of
$H_s$. Denote $W_{i,s}:=s \int_{\R^d} w_s(x)^i \diff x$, $i\in\N$.
Applying the Mecke formula in the first equality, we obtain
\begin{align*}
	\Var(H_s) &=s\int_{\R^d} w_s(x)^2
	e^{-k_d s \rho_s^d} \diff x- \left(s \int_{\R^d}w_s(x)
	e^{-k_d s \rho_s^d} \diff x\right)^2\\
	&\qquad + s^2 \int_{\R^d} \int_{B(x,\rho_s)^c}
	w_s(x)w_s(y)
	\exp\left\{-\operatorname{Vol}(B(x,\rho_s) 
	\cup B(y,\rho_s))\right\} \diff y \diff x \\
	& \ge  e^{-k_d s \rho_s^d}W_{2,s} -  s^2 e^{-2k_d s
		\rho_s^d}\int_{\R^d} \int_{\R^d \cap B(x,\rho_s)}
	w_s(x)w_s(y) \diff y\diff x.
\end{align*}
As in the previous example,
\begin{displaymath}
	s^2 \int_{\R^d} \int_{\R^d \cap B(x,\rho_s)}
	w_s(x)w_s(y) \diff y \diff x \le k_d s\rho_s^d W_{2,s},
\end{displaymath}
so that 
\begin{displaymath}
	\Var(H)  \ge e^{-k_d s \rho_s^d}(1- k_d s \rho_s^d e^{-k_d s
		\rho_s^d}) 
	W_{2,s} \ge \frac{1}{2} e^{-k_ds \rho_s^d}W_{2,s},
\end{displaymath}
where in the last step we have used that $ue^{-u} \le 1/2$ for $u \ge 0$.
Denoting $\bar w_s:=\max\{w_s^2,w_s^4, w_s^{9/2}\}$, it is straightforward to
check that
\begin{displaymath}
	f_\alpha (y) \le C s e^{-\alpha k_d s \rho_s^d}
	\int_{B(y,3\rho_s)} \bar w_s(x) \diff x
\end{displaymath}
for $\alpha>0$ and a constant $C>0$, so that by Jensen's inequality,
\begin{displaymath}
	f_\alpha (y)^2 \le C^2 3^d k_d s^2 \rho_s^d e^{-2\alpha k_d s
		\rho_s^d} 
	\int_{B(y,3\rho_s)} \bar w_s(x)^2 \diff x.
\end{displaymath}
Thus,
letting $\overline W_{i,s}:=s\int_{\R_d} \bar w_s(x)^i \diff x$, $i \in \N$, and
$\beta=1/36$, and using again that $ue^{-u} \le 1/2$ for $u \ge 0$,
we have that there exists a constant $C_d$ depending only on the
dimension $d$ such that
\begin{displaymath}
	s\QQ f_\beta^2 \le C_d  \overline W_{2,s},\quad\text{ and }\quad
	\max\{s\QQ f_{2\beta}, s\QQ ((\kappa_s+g_s)^{2\beta}G_s) \}
	\le C_d \overline W_{1,s}.
\end{displaymath}
Thus, applying Theorem~\ref{thm:KolBd}, we obtain for $s \ge 1$ that
\begin{align*}
	\max&\left\{d_{W}\left(\frac{H_s-\E H_s}{\sqrt{\Var H_s}},N\right), d_{K}\left(\frac{H_s-\E H_s}{\sqrt{\Var H_s}},N\right) \right\}\\
	&\qquad \leq C'_d \Bigg[\frac{\overline W_{2,s}^{1/2}
		+ \overline W_{1,s}^{1/2}}{e^{-k_d s \rho_s^d} W_{2,s}}
	+\frac{\overline W_{1,s}}{(e^{-k_d s \rho_s^d} W_{2,s})^{3/2}}
	+\frac{\overline W_{1,s}^{5/4}
		+ \overline W_{1,s}^{3/2}}{(e^{-k_d s \rho_s^d} W_{2,s})^{2}}\Bigg]
\end{align*}
for some constant $C'_d>0$ depending only on the dimension. The
setting can be easily extended for functions $\rho_s$ which depend
on the position $x$ (see \cite{iyer:thac12}) and/or are random
variables which, together with locations, form a Poisson process
on the product space.

As an example, consider the logarithmic weight function
$w_s(x):=\log \frac{s}{\|x\|} \mathds{1}_{x \in B(0,s)}$. 
For $i \in \N$,
\begin{displaymath}
	W_{i,s}=s \int_{B(0,s)} \log^i \frac{s}{\|x\|} \diff x
	=dk_d s \int_0^s r^{d-1} \log^i \frac{s}{r} \diff r
	= dk_d s^{d+1} \int_0^1 z^{d-1} \log^i \frac{1}{z} \diff z
	= \mathcal{O}(s^{d+1}),
\end{displaymath}
so that $\overline W_{i,s}=\mathcal{O}(s^{d+1})$ for all $i \in \N$.
Hence, in the regime when
$s\rho_s^d-(d+1)(2k_d)^{-1}\log s \to -\infty$ as $s \to \infty$,
one obtains Gaussian convergence as $s \to \infty$ with an
appropriate non-asymptotic bound on the Wasserstein or Kolmogorov
distances between the normalized $H_s$ and a standard normal random
variable $N$.
\end{example}

\section{Modified bounds on the Wasserstein and Kolmogorov distances and proof of Theorem~\ref{thm:KolBd}}\label{sec:Proof}

In this section, we prove Theorem~\ref{thm:KolBd}. The proof is
primarily based on the following generalization of Theorem~6.1 in
\cite{LPS16}, incorporating a spatially inhomogeneous moment bound
given by a function $c_x$, $x \in \XX$. The proof, which we present
for completeness in the Appendix follows closely that of
\cite[Theorem~6.1]{LPS16}.

Let $\sP$ be a Poisson
process on a measurable space $(\XX, \F)$ with a $\sigma$-finite
intensity measure $\nu$. Let $F:=f(\sP)$ be a measurable function of $\sP$. For $x,y \in \XX$,
define the first and second order difference operators as
$D_x F:=f(\sP+\delta_x)-f(\sP)$ and $D_{x,y}^2 F :=D_x(D_y F)$.  Also,
denote by $\operatorname{dom} D$ the collection of functions
$F \in L_\sP^2$ with
\begin{displaymath}
\E \int_\XX \left(D_{x} F\right)^{2} \nu(\ldiff x)<\infty.
\end{displaymath}

\begin{theorem}\label{thm:Main}
Let $F \in \operatorname{dom} D$ be such that $\Var F>0$.  Assume
that there exists a $q>0$ such that, for all  $\mu\in\Nb$ with $\mu(\XX) \le 1$,
\begin{displaymath}
	\E\left|D_{x} F(\sP+\mu)\right|^{4+q} 
	\leq c_{x} \quad \text{for $\nu$-a.e.\;} x \in \XX,
\end{displaymath}
where $c_x$ is a measurable function of $x\in\XX$.
Then
\begin{multline*}
	d_{W}\left(\frac{F-\E F}{\sqrt{\Var F}}, N\right)\\
	\ \leq\; \frac{12}{\Var F}\left[\int_\XX\left(\int_\XX c_{x_1}^{2
		/\left(4+q\right)}
	\Prob{D_{x_{1}, x_{2}}^{2} F \neq 0}^{q /\left(16+4 q\right)} 
	\nu(\ldiff x_{1})\right)^{2} \nu(\ldiff x_{2})\right]^{1 / 2}
	+\frac{\Gamma_{F}}{(\Var F)^{3
			/ 2}},
\end{multline*}
and
\begin{align*}
	d_{K}\left(\frac{F-\E F}{\sqrt{\Var F}}, N\right) \leq\; 
	& \frac{12}{\Var F}\left[\int_\XX\left(\int_\XX c_{x_1}^{2
		/\left(4+q\right)}
	\Prob{D_{x_{1}, x_{2}}^{2} F \neq 0}^{q /\left(16+4 q\right)} 
	\nu(\ldiff x_{1})\right)^{2} \nu(\ldiff x_{2})\right]^{1 / 2} \\
	&+\frac{\Gamma_{F}^{1 / 2}}{\Var F}+\frac{2\Gamma_{F}}{(\Var F)^{3
			/ 2}}
	+\frac{\Gamma_{F}^{5 / 4}+2 \Gamma_{F}^{3 / 2}}{(\Var F)^{2}} \\
	&+\frac{12}{\Var F}\left[\int_{\XX^2} c_{x_1}^{4
		/\left(4+q\right)}
	\Prob{D_{x_{1}, x_{2}}^{2} F \neq 0}^{q /\left(8+2 q\right)} 
	\nu^{2}(\ldiff (x_{1}, x_{2}))\right]^{1 / 2},
\end{align*}
with 
\begin{displaymath}
	\Gamma_{F}:=\int_\XX \max\{c_x^{2/(4+q)},c_x^{4/(4+q)}\}
	\Prob{D_{x} F \neq 0}^{q /\left(8+2 q\right)} \nu(\ldiff x).
\end{displaymath}
\end{theorem} 

For a proof of this result, see the Appendix.  We derive
Theorem~\ref{thm:KolBd} from Theorem~\ref{thm:Main} by
proving a series of lemmas, following the general structure of the
proof of Theorem~2.1(a) in \cite{LSY19}. However, our setting is more
versatile, enabling us to handle new examples. The first lemma is an
exact restatement of \cite[Lemma~5.2]{LSY19}, which is also
contained in Remark~6.2 of \cite{LPS16}. Recall the definition of
$H_s$ given at \eqref{eq:hs}.

\begin{lemma}\label{lem:D} 
For $s \ge 1$, $\M \in \Nb$ and $y_1, y_2, y_3 \in \XX$, 
\begin{displaymath}
	D_{y} H_{s}(\M)=\xi_{s}(y, \M +\delta_y)+\sum_{x \in \M} D_{y} \xi_{s}(x, \M)
\end{displaymath}
and
\begin{displaymath}
	D_{y_{1}, y_{2}}^{2} H_{s}(\M)= D_{y_{1}} \xi_{s}
	\left(y_{2}, \M +\delta_{y_{2}}\right)
	+D_{y_{2}} \xi_{s}\left(y_{1}, \M+\delta_{y_{1}}\right)
	+\sum_{x \in \M} D_{y_{1}, y_{2}}^{2} \xi_{s}(x, \M).
\end{displaymath}
\end{lemma}

The next lemma shows that the difference
operator $D_y$ vanishes if $y$ lies outside the stabilization region.

\begin{lemma}\label{lem:Dnull}
Assume that (A1) holds and let $\M \in \Nb$ and
$x,y, y_1, y_2 \in \XX$. Then for $s \ge 1$,
\begin{displaymath}
	D_y \xi_s (x,\M +\delta_x)=0 \, \text{ if }\, y \not \in R_s(x, \M+\delta_x),
\end{displaymath}
and
\begin{displaymath}
	D_{y_1,y_2}^2 \xi_s(x,\M+\delta_x)=0
	\text{ if }\, \{y_1,y_2\} \not \subseteq R_s(x, \M+\delta_x).
\end{displaymath}
\end{lemma}

\begin{proof} 
By (A1.4),
\begin{align*}
	D_y \xi_s (x,\M+\delta_x)
	&=\xi_s (x,\M+\delta_x+\delta_y) - \xi_s (x,\M+\delta_x)\\
	&=\xi_s \Big(x,(\M+\delta_x+\delta_y)_{R_s(x,\M+\delta_x+\delta_y)}\Big)
	- \xi_s\Big(x,(\M+\delta_x)_{R_s(x,\M+\delta_x)}\Big),
\end{align*}
If $(\M+\delta_x)_{R_s(x,\M+\delta_x)}=0$, by the monotonicity
property (A1.2), for $y \notin R_s(x,\M+\delta_x)$ we have
$(\M+\delta_x+\delta_y)_{R_s(x,\M+\delta_x+\delta_y)}=0$ yielding
$D_y \xi_s (x,\M+\delta_x)=0$. If
$(\M+\delta_x)_{R_s(x,\M+\delta_x)}\neq 0$, then (A1.3) implies that $(\M+\delta_x+\delta_y)_{R_s(x,\M+\delta_x+\delta_y)}\neq 0$. Thus, for $y \not \in R_s(x, \M+\delta_x)$, by (A1.4) and
\eqref{eq:ximon} we have
\begin{displaymath}
	\xi_s
	\Big(x,(\M+\delta_x+\delta_y)_{R_s(x,\M+\delta_x+\delta_y)}\Big)
	=\xi_s \Big(x,(\M+\delta_x+\delta_y)_{R_s(x,\M+\delta_x)}\Big)=\xi_s \Big(x,(\M+\delta_x)_{R_s(x,\M+\delta_x)}\Big),
\end{displaymath}
so that
$D_y \xi_s (x,\M+\delta_x)$ vanishes. 

Finally, by (A1.2),
$y_1 \not \in R_s(x, \M+\delta_x)$ implies
$y_1 \not \in R_s(x, \M+\delta_{y_2}+\delta_x)$. Hence, the second
order difference operator vanishes, being an iteration of the first
order one. If $y_2 \not \in R_s(x, \M+\delta_x)$, a similar argument
applies.
\end{proof}

The next lemma, which is similar to \cite[Lemma~5.4(a)]{LSY19}
provides a bound in terms of $M_s$ on the $(4+\e)$-th moment
of the difference operator for any $\e \in (0,p]$, where $p \in (0,1]$
and $M_s$ are as in (A2). 

\begin{lemma}\label{lem:Dmom}
Assume that (A2) holds. For all $\e \in (0,p]$, $s \ge 1$,
$x,y \in \XX$ and $\mu\in\Nb$ with $\mu(\XX) \le 6$
\begin{displaymath}
	\E\Big|D_y \xi_{s}\big(x, \sP_{s}+\delta_x +\mu\big)\Big|^{4+\e}
	\leq 2^{4+\e} M_s(x)^{4+\e}.
\end{displaymath}
\end{lemma}
\begin{proof} 
By Jensen's inequality, H\"older's inequality and assumption (A2), 
\begin{align*}
	&\E\Big|D_y \xi_{s}\big(x, \sP_{s}+\delta_x+\mu\big)\Big|^{4+\varepsilon}\\
	& \le 2^{3+\e} \E \left(\left|\xi_{s}\left(x, \sP_{s} 
	+\delta_x+\delta_y+\mu\right)\right|^{4+\varepsilon}
	+\left|\xi_{s}\left(x, \sP_{s} +\delta_x+\mu\right)\right|^{4+\varepsilon}\right) 
	\le 2^{4+\varepsilon} M_s(x)^{4+\varepsilon}.  \qedhere
\end{align*}
\end{proof}

Recall the functions $g_s$ and $h_s$ defined at \eqref{eq:g}. 

\begin{lemma}\label{lem:1MomD}
Assume that (A1) and (A2) hold. Then, there exists a constant $C_{p} \in
[1,\infty)$ depending only on $p$, such that
\begin{displaymath}
	\E \Big|D_y H_s(\sP_s +\mu)\Big|^{4+p/2}
	\le C_{p} \left[M_s^{4+p/2}(y) + h_s(y)(1+g_s(y)^4) \right]
\end{displaymath}
for all $y \in \XX$, $\mu\in\Nb$ with $\mu(\XX)\le 1$, and
$s \ge 1$.
\end{lemma}

\begin{proof}
Let $\e:=p/2$. We argue as in \cite{LSY19}. For $\mu=0$, using
Lemma~\ref{lem:D} followed by Jensen's inequality,
\begin{align*}
	\E\left|D_{y} H_{s}(\sP_{s})\right|^{4+\e}
	&=\E\left|\xi_{s}(y, \sP_{s}+\delta_y)
	+\sum_{x \in \sP_{s}} D_{y} \xi_{s}(x, \sP_{s})\right|^{4+\e} \\
	&\leq 2^{3+\e} \E\left|\xi_{s}(y, \sP_{s}+\delta_y)\right|^{4+\e}+2^{3+\e}
	\E\left|\sum_{x \in \sP_{s}} D_{y} \xi_{s}(x,\sP_{s})\right|^{4+\e}.
\end{align*}
By (A2), the first summand is bounded by
$2^{3+\e}M_s(y)^{4+\eps}$. Following the argument in
\cite[Lemma~5.5]{LSY19}, the second summand can be bounded as
\begin{displaymath}
	2^{3+\e} \E\left|\sum_{x \in \sP_{s}} D_{y} \xi_{s}
	\left(x, \sP_{s}\right)\right|^{4+\e}
	\le 2^{3+\e} (I_1+ 15 I_2 + 25 I_3 + 10 I_4 + I_5),
\end{displaymath}
where for $i \in \{1,\dots,5\}$,
\begin{displaymath}
	I_{i}=\E \sum_{(x_{1},\dots,x_{i})\in \sP_{s}^{i,\neq}}
	\mathds{1}_{D_{y}\xi_{s}(x_{j}, \sP_{s})\neq 0, j=1,\dots,i}
	\big|D_{y} \xi_{s}(x_{1}, \sP_{s})\big|^{4+\e}.
\end{displaymath}
Here $\sP_{s}^{i, \neq}$ stands for the set of all $i$-tuples of distinct
points from $\sP_{s}$, where multiple points at the same
location are considered to be different ones. Applying the
multivariate Mecke formula in the first equation, H\"older's
inequality followed by Lemma~\ref{lem:Dmom} in the second step and
Lemma~\ref{lem:Dnull} and (A1.2) in the third step, we obtain for $1 \le i \le 5$,
\begin{align*}
	I_i &=s^{i} \int_{\XX^{i}}\E\Big[\mathds{1}_{D_{y}
		\xi_{s}(x_{j}, \sP_{s}+\delta_{x_1}+\cdots+\delta_{x_i})\neq 0, j=1,
		\dots, i} 
	\big|D_y \xi_s (x_1, \sP_s+\delta_{x_1}+\cdots+\delta_{x_i})\big|^{4+\e} \Big] \QQ^{i}(\ldiff(x_{1},\dots,x_{i}))\\
	&\le s^{i} \int_{\XX^{i}} (2M_s(x_1))^{4+\e} 
	\prod_{j=1}^{i} \Prob{D_y \xi_s (x_j, \sP_s+\delta_{x_1}+\cdots
		+\delta_{x_i})\neq0}^{\frac{p-\e}{4 i+p i}}\;\QQ^{i}(\ldiff(x_{1},\dots,x_{i}))\\
	& \le 2^{4+\e} s^{i} \int_{\XX^{i}} M_s(x_1)^{4+\e} 
	\prod_{j=1}^{i} \Prob{y \in R_s(x_j,\sP_s+\delta_{x_j})}^{\frac{p-\e}{4 i+p i}}
	\;\QQ^{i}(\ldiff(x_{1},\dots,x_{i})).
\end{align*}
By \eqref{eq:Rs},
\begin{align*}
	&2^{-4-\e}I_i 
	\le \, s^{i} \int_{\XX^{i}} M_s(x_1)^{4+\e}
	\prod_{j=1}^{i} \exp\Big\{-\frac{p-\e}{4 i+p i}
	r_{s}(x_j,y)\Big\} \QQ^{i}(\ldiff(x_{1},\dots,x_{i}))\\
	& \,= \left(s\int_{\XX} \exp \Big\{-\frac{p-\e}{4
		i+p i} r_{s}(x, y)\Big\} \QQ(\ldiff x)\right)^{i-1}  \left(s\int_{\XX} M_s(x)^{4+\e} \exp \Big\{-\frac{p-\e}{4
		i+p i} r_{s}(x, y)\Big\} \QQ(\ldiff x)\right)\\
	&\, \le \left(s\int_{\XX} \exp
	\Big\{-\frac{p}{40+10 p} r_{s}(x, y)\Big\} 
	\QQ(\ldiff x)\right)^{i-1}  \left(s\int_{\XX} M_s(x)^{4+\e} \exp \Big\{-\frac{p}{40+10 p} r_{s}(x, y)\Big\} \QQ(\ldiff x)\right) \\
	&\,\le g_{s}(y)^{i-1} h_s(y),
\end{align*}
where $g_{s}$ and $h_s$ are defined at \eqref{eq:g}. Since
$g_{s}^{i-1} \le 1+g_{s}^4$ for all $i=1,\dots,5$, this proves the
result for $\mu=0$. If $\mu(\XX)=1$, the proof is similar,
see the proof of \cite[Lemma~5.5]{LSY19} for details.
\end{proof}

\begin{lemma}\label{lem:inner}
Assume that (A1) holds. For any $\beta>0$,
$s \ge 1$ and $x_2 \in \XX$,
\begin{displaymath}
	s \int_{\XX} G_s(x_1)\Prob{D_{x_1,x_2}^{2} H_{s}(\sP_{s})
		\neq 0}^{\beta}\QQ(\ldiff x_{1}) \leq 3^\beta f_\beta(x_2)
\end{displaymath}
with $f_\beta$ defined at \eqref{eq:fa}.
\end{lemma}

\begin{proof}
As in the proof of \cite[Lemma~5.9(a)]{LSY19}, by Lemma~\ref{lem:D}
and the Mecke formula, one has
\begin{equation}
	\label{eq:2Dsum}
	\Prob{D_{x_1,x_2}^{2} H_{s}(\sP_{s}) \neq 0} 
	\le \Prob{D_{x_1} \xi_{s}(x_2,\sP_{s}+\delta_{x_2})\neq 0}
	+\Prob{D_{x_2} \xi_{s}(x_1,\sP_{s}+\delta_{x_1}) \neq 0} + T_{x_1,x_2,s},
\end{equation}
where 
\begin{displaymath}
	T_{x_1,x_2,s}:= s \int_{\XX} \Prob{D_{x_1,x_2}^{2}
		\xi_{s}(z,\sP_{s}+\delta_z) \neq 0} \QQ(\ldiff z).
\end{displaymath}
By Lemma~\ref{lem:Dnull} and \eqref{eq:Rs}, the first two summands
on the right-hand side of \eqref{eq:2Dsum} are bounded by
$e^{- r_{s}(x_2,x_1)}$ and $e^{- r_{s}(x_1,x_2)}$,
respectively. Furthermore, by Lemma~\ref{lem:Dnull} and
\eqref{eq:g2s},
\begin{displaymath}
	T_{x_1,x_2,s} \le s \int_{\XX} \Prob{\{x_1,x_2\}
		\subseteq R_s(z,\sP_s+\delta_z)} \QQ(\ldiff z)
	= q_{s}(x_1,x_2).
\end{displaymath}
By \eqref{eq:fal},
\begin{align*}
	& s  \int_{\XX}  G_s(x_1)\Prob{D_{x_1,x_2}^{2}H_{s}(\sP_{s})\neq
		0}^{\beta}\QQ(\ldiff x_{1})\\
	&\le 3^\beta \int_{\XX}  G_s(x_1) \left[ e^{- \beta
		r_{s}(x_2,x_1)} +e^{- \beta r_{s}(x_1,x_2)}
	+ q_{s}(x_1,x_2)^\beta\right] \QQ(\ldiff x_{1}) = 3^\beta f_\beta(x_2). \qedhere
\end{align*}
\end{proof}

Recall the function $\kappa_s(x)$ in \eqref{eq:p}.

\begin{lemma}\label{lem:bounds}
Assume that (A1) holds, and let $\beta >0$. Then for all $s
\ge 1$,
\begin{gather*}
	s \int_{\XX}\left(s \int_{\XX} G_s(x_1)\Prob{ D_{x_1,x_2}^{2} H_{s}
		(\sP_{s}) \neq 0}^{\beta} \QQ(\ldiff x_{1})\right)^{2} 
	\QQ(\ldiff x_{2}) \le s3^{2\beta} \QQ f_\beta^2,\\
	s^{2} \int_{\XX^{2}} G_s(x_1)\Prob{ D_{x_1,x_2}^{2}
		H_{s}(\sP_{s})\neq 0}^{\beta} \QQ^{2}(\ldiff (x_{1}, x_{2}))
	\le s 3^{\beta} \QQ f_\beta,\\
	s \int_{\XX} G_s(x)\Prob{ D_{x} H_{s}(\sP_{s})\neq 0}^{\beta}
	\QQ(\ldiff x) \le s \QQ ((\kappa_s+g_{s})^\beta G_s).
\end{gather*}
\end{lemma}

\begin{proof}
The first two assertions follow directly from
Lemma~\ref{lem:inner}. For the last one, by Lemma~\ref{lem:D} and
the Mecke formula, we can write
\begin{align*}
	\Prob{D_{x} H_{s}(\sP_{s}) \neq 0} 
	\leq \Prob{\xi_{s}(x,\sP_{s}+\delta_x)\neq 0}
	+\E \sum_{z \in \sP_{s}} \mathds{1}_{D_{x} \xi_{s}(z,\sP_{s})\neq 0}\\
	= \kappa_s(x)+s \int_{\XX} \Prob{D_{x} \xi_{s}(z,\sP_{s}+\delta_z)\neq 0} 
	\QQ(\ldiff z) \le \kappa_s(x)+g_{s}(x),
\end{align*}
where we used Lemma~\ref{lem:Dnull}, \eqref{eq:Rs} and
\eqref{eq:g} in the final step. This yields the final assertion.
\end{proof}

\begin{proof}[{Proof of Theorem~\ref{thm:KolBd}:}]
In view of Lemma~\ref{lem:1MomD}, the condition in
Theorem~\ref{thm:Main} is satisfied with the exponent $4+p/2$ with
$c_y:= C_p \left[M_s(y)^{4+p/2} + h_s(y)(1+g_s(y)^4) \right]$ for $y \in \XX$. Hence,
\begin{align*}
	&\max\left\{c_{y}^{2/(4+p/2)},c_{y}^{4/(4+p/2)}\right\}\\
	&\le  C_p^{4/(4+p/2)} \left[\max\left\{M_s(y)^2,M_s(y)^4 \right\} +
	\max\{h_s(y)^{2/(4+p/2)}, h_s(y)^{4/(4+p/2)}\}\big(1+g_s(y)^4\big) \right]\\
	&= C_p^{4/(4+p/2)} G_s(y),
\end{align*}
where $G_s$ is defined at \eqref{eq:g5}. The result now
follows from Theorem~\ref{thm:Main} upon using
Lemma~\ref{lem:bounds}.
\end{proof}

\section*{Acknowledgements}
We would like to thank Larry Goldstein for pointing out the work \cite{EG81} to provide lower bounds, and Matthias Schulte for many helpful discussions that vastly improved the presentation of the paper. We are also grateful to Joe Yukich and Giovanni Peccati for
their helpful comments on the manuscript.


\begin{thebibliography}{10}

\bibitem{bai:dev:hwan:05}
Bai, Z. D., Devroye, L., Hwang, H. K. and Tsai, T. H.: 
\newblock Maxima in hypercubes.
\newblock {\em Random Struct. Algorithms}, \textbf{27}, (2005), 290--309.

\bibitem{BX06}
Barbour, A.~D. and Xia, A.:
\newblock Normal approximation for random sums.
\newblock {\em Adv. in Appl. Probab.}, \textbf{38}, (2006), 693--728.

\bibitem{Ba20}
Baryshnikov, Y.:
\newblock Supporting-points processes and some of their applications.
\newblock {\em Probab. Theory Related Fields}, \textbf{117}, (2000), 163--182.

\bibitem{Bha21}
Bhattacharjee, C.:
\newblock Gaussian approximation in random minimal directed spanning trees.
\newblock {\em Random Struct. Algorithms}, \textbf{61}, (2022), 462--492.

\bibitem{bhat-mol21}
Bhattacharjee, C., Molchanov,  I. and Turin, R.:
\newblock Central limit theorem for birth-growth model with Poisson arrivals
and random growth speed.
\newblock {\em arXiv preprint arXiv:2107.06792}, (2021).

\bibitem{EG81}
Gunnar, E.:
\newblock A remainder term estimate for the normal approximation in classical occupancy.
\newblock {\em Ann. Probab.}, \textbf{9}, (1981), 684--692.

\bibitem{fil:naim20}
Fill, J.~A. and Naiman, D.~Q.:
\newblock The {P}areto record frontier.
\newblock {\em Electron. J. Probab.}, \textbf{25}, (2020).

\bibitem{iyer:thac12}
Iyer, S.~K. and Thacker, D.:
\newblock Nonuniform random geometric graphs with location-dependent radii.
\newblock {\em Ann. Appl. Probab.}, \textbf{22}, (2012), 2048--2066.

\bibitem{LSY19}
Lachi\`eze-Rey, R., Schulte, M. and Yukich, J.~E.:
\newblock Normal approximation for stabilizing functionals.
\newblock {\em Ann. Appl. Probab.}, \textbf{29}, (2019), 931--993.

\bibitem{LPS16}
Last, G., Peccati, G. and Schulte, M.:
\newblock Normal approximation on {P}oisson spaces: {M}ehler's formula, second
order {P}oincar\'{e} inequalities and stabilization.
\newblock {\em Probab. Theory Related Fields}, \textbf{165}, (2016), 667--723.

\bibitem{last:pen}
Last, G. and Penrose, M.:
\newblock {\em Lectures on the {Poisson} Process}.
\newblock Cambridge Univ. Press, Cambridge, 2018.

\bibitem{pen03}
Penrose, M.:
\newblock {\em Random Geometric Graphs}.
\newblock Oxford University Press, Oxford, 2003.

\bibitem{PY01}
Penrose, M.~D. and Yukich, J.~E.:
\newblock Central limit theorems for some graphs in computational geometry.
\newblock {\em Ann. Appl. Probab.}, \textbf{11}, (2001), 1005--1041.

\bibitem{pen:yuk03}
Penrose, M.~D. and Yukich, J.~E.:
\newblock Weak laws of large numbers in geometric probability.
\newblock {\em Ann. Appl. Probab.}, \textbf{13}, (2003), 277--303.

\bibitem{pen:yuk05}
Penrose, M.~D. and Yukich, J.~E.:
\newblock Normal approximation in geometric probability.
\newblock In {\em Stein's Method and Applications}, volume~5 of {\em Lect.
	Notes Ser. Inst. Math. Sci. Natl. Univ. Singap.}, 37--58. Singapore
Univ. Press, Singapore, 2005.

\bibitem{sch10}
Schreiber, T.:
\newblock Limit theorems in stochastic geometry.
\newblock In W.~S. Kendall and I.~Molchanov, editors, {\em New Perspectives in
	Stochastic Geometry}, pages 111--144. Oxford Univ. Press, Oxford, 2010.

\bibitem{Yu15}
Yukich, J.~E.:
\newblock Surface order scaling in stochastic geometry.
\newblock {\em Ann. Appl. Probab.}, \textbf{25}, (2015), 177--210.

\end{thebibliography}

\section*{Appendix : Proof of Theorem~\ref{thm:Main}}
\label{sec:modKol}

In this section, we prove Theorem~\ref{thm:Main}, which is a
slightly modified version of Theorem~6.1 in
\cite{LPS16}. Recall that $\sP$ is a Poisson process on a measurable
space $(\XX, \F)$ with a $\sigma$-finite intensity measure $\nu$ and
$F:=f(\sP)$ is a measurable function of $\sP$. For $x,y \in \XX$,
recall the definitions of the first and second order difference
operators $D_x F$ and $D_{x,y}^2 F$ and that of $\operatorname{dom} D$
from Section~\ref{sec:Proof}.

We are generally interested in the Gaussian approximation of such a
function $F$ with zero mean and unit variance with the aim to
bound the Wasserstein and the Kolmogorov distances between $F$ and a
standard normal random variable $N$. An important result in this
direction was given in \cite{LPS16}. Define
\begin{align*}
	\gamma_{1} &:=4\left[\int_{\XX^3} 
	\left[\E\left(D_{x_{1}} F\right)^{2}\left(D_{x_{2}}
	F\right)^{2}\right]^{1/2}
	\left[\E\left(D_{x_{1}, x_{3}}^{2} F\right)^{2}\left(D_{x_{2},
		x_{3}}^{2} F\right)^{2}\right]^{1/2} 
	\nu^{3}(\ldiff (x_{1}, x_{2}, x_{3}))\right]^{1 / 2}, \\
	\gamma_{2} &:=\left[\int_{\XX^3} 
	\E \left[\left(D_{x_{1}, x_{3}}^{2} F\right)^{2}\left(D_{x_{2},
		x_{3}}^{2} F\right)^{2} \right]
	\nu^{3}(\ldiff (x_{1}, x_{2}, x_{3}))\right]^{1 / 2}, \\
	\gamma_{3} &:=\int_\XX \E\left|D_{x} F\right|^{3} \nu(\ldiff x),\\
	\gamma_{4} &:=\frac{1}{2}\left[\E F^{4}\right]^{1 / 4} 
	\int_\XX \left[\E\left(D_{x} F\right)^{4}\right]^{3 / 4} \nu(\ldiff x), \\
	\gamma_{5} &:=\left[\int_\XX \E\left(D_{x} F\right)^{4} \nu(\ldiff x)\right]^{1 / 2}, \\
	\gamma_{6} &:=\left[\int_{\XX^2} \left(6\left[\E\left(D_{x_{1}}
	F\right)^{4}\right]^{1 / 2}
	\left[\E\left(D_{x_{1}, x_{2}}^{2} F\right)^{4}\right]^{1 / 2}
	+3 \E\left(D_{x_{1}, x_{2}}^{2} F\right)^{4} \right) 
	\nu^{2}(\ldiff(x_{1}, x_{2}))\right]^{1 / 2}.
\end{align*}

\begin{theorem*}[\cite{LPS16}, Theorems~1.1 and 1.2] \label{thm:Schulte}
	For $F \in \operatorname{dom} D$ having zero mean and unit variance,
	\begin{displaymath}
	d_{W}(F, N) \leq \gamma_{1}+\gamma_{2}+\gamma_{3},
	\end{displaymath}
	and
	\begin{displaymath}
	d_{K}(F, N) \leq \gamma_{1}+\gamma_{2}+\gamma_{3}+\gamma_{4}+\gamma_{5}+\gamma_{6}.
	\end{displaymath}
\end{theorem*}

Under additional assumptions on the difference operator, one can
simplify the bound. This is done in \cite[Theorem~6.1]{LPS16},
assuming that, for some $q>0$, the $(4+q)$-th moment of the difference
operator $D_{x} F(\sP + \mu)$ for $\mu\in\Nb$ with total mass at most
one is uniformly bounded in $x \in \XX$. However, in some
applications, as is the case in the example of minimal points
discussed in Section~\ref{sec:Pareto}, such a uniform bound does not
exist. In Theorem~\ref{thm:Main}, we modify \cite[Theorem~6.1]{LPS16}
to allow for a non-uniform bound depending on $x$. Below, we present
the proof of Theorem~\ref{thm:Main} for completeness, though the
arguments remain largely similar to those in the proof of Theorem~6.1
in \cite{LPS16}, with the main difference being the presence of a
spatially inhomogeneous moment bound given by the function $c_x$.

\begin{proof}[Proof of Theorem~\ref{thm:Main}]
	By our assumption, H\"older's inequality yields
	that
	\begin{displaymath}
		\E\left(D_{x} F\right)^{4} \leq 
		\left[\E\left|D_{x} F\right|^{4+q}\right]^{4 /\left(4+q\right)} \Prob{D_{x} F \neq 0}^{q
			/\left(4+q\right)}
		\leq c_{x}^{4 /\left(4+q\right)} \Prob{D_{x} F \neq 0}^{q /\left(4+q\right)}
	\end{displaymath}
	and
	\begin{displaymath}
		\E\left|D_{x} F\right|^{3} \leq c_{x}^{3 /\left(4+p\right)} 
		\Prob{D_{x} F \neq 0}^{\left(1+q\right) /\left(4+q\right)}.
	\end{displaymath}
	Also, using H\"older's inequality as above and Jensen's inequality
	in the second step, we have
	\begin{align*}
		\E\left(D_{x_{1}, x_{2}}^{2} F\right)^{4} 
		& \leq \left[\E\left|D_{x_{1}, x_{2}}^{2} F\right|^{4+q}\right]^{4 /\left(4+q\right)} \Prob{D_{x_{1}, x_{2}}^{2} F \neq 0}^{q /\left(4+q\right)}\\
		& \leq 16 \min\{c_{x_1},c_{x_2}\}^{4 /\left(4+q\right)}
		\Prob{D_{x_{1}, x_{2}}^{2} F \neq 0}^{q /\left(4+q\right)}.
	\end{align*}
	Thus, evaluating $(\gamma_i)_{1\le i \le 6}$ for
	$(F-\E F)/\sqrt{\Var F}$, we obtain
	\begin{align*}
		\gamma_{1} &\leq \frac{8}{\Var F}\Bigg[\int_{\XX^3}  
		c_{x_1}^{2 /\left(4+q\right)} c_{x_2}^{2 /\left(4+q\right)} 
		\Prob{D_{x_{1}, x_{3}}^{2} F \neq 0}^{q /\left(16+4q\right)}\\
		&\qquad \qquad\qquad\qquad\qquad\qquad \;\quad \times \Prob{D_{x_{2}, x_{3}}^{2} F \neq 0}^{q /\left(16+4
			q\right)} 
		\nu^{3}(\ldiff (x_{1}, x_{2}, x_{3}))\Bigg]^{1 / 2} \\
		&=\frac{8}{\Var F}\left[\int_\XX \left(\int_\XX c_{x_1}^{2
			/\left(4+q\right)} 
		\Prob{D_{x_{1}, x_{2}}^{2} F \neq 0}^{q /\left(16+4 q\right)}
		\nu(\ldiff x_{1})\right)^2 
		\nu(\ldiff x_{2})\right]^{1 / 2},\\
		\gamma_{2} &\leq \frac{4}{\Var F}\Bigg[\int_{\XX^3} 
		c_{x_1}^{2 /\left(4+q\right)}c_{x_2}^{2 /\left(4+q\right)} 
		\Prob{D_{x_{1}, x_{3}}^{2} F \neq 0}^{q /\left(8+2
			q\right)} \\
		&\qquad \qquad\qquad\qquad\qquad\qquad \;\quad \times \Prob{D_{x_{2}, x_{3}}^{2} F \neq 0}^{q /\left(8+2
			q\right)}\nu^{3}(\ldiff(x_{1}, x_{2}, x_{3}))\Bigg]^{1 / 2}\\
		&\le \frac{4}{\Var F}\left[\int_\XX \left(\int_\XX c_{x_1}^{2
			/\left(4+q\right)} 
		\Prob{D_{x_{1}, x_{2}}^{2} F \neq 0}^{q /\left(16+4 q\right)} 
		\nu(\ldiff x_{1})\right)^2 \nu(\ldiff x_{2})\right]^{1 / 2},\\
		\gamma_{3} &\leq \frac{1}{(\Var F)^{3 / 2}} 
		\int_\XX c_{x}^{3 /\left(4+q\right)} 
		\Prob{D_{x} F \neq 0}^{\left(1+q\right) /\left(4+q\right)} 
		\nu(\ldiff x) \le \frac{\Gamma_{F}}{(\Var F)^{3 / 2}},\\
		\gamma_{4} &\leq \frac{1}{2(\Var F)^{2}}\left[\E(F-\E
		F)^{4}\right]^{1 / 4} 
		\int_\XX c_{x}^{3 /\left(4+q\right)} \Prob{D_{x} F \neq 0}^{q
			/\left(8+2 q\right)} \nu(\ldiff x)\\
		&\leq \frac{\Gamma_F}{2(\Var F)^{2}}\left[\E(F-\E F)^{4}\right]^{1 / 4},\\
		\gamma_{5} &\leq \frac{1}{\Var F}\left[\int_\XX c_{x}^{4
			/\left(4+q\right)} 
		\Prob{D_{x} F \neq 0}^{q /\left(4+q\right)} \nu(\ldiff
		x)\right]^{1 / 2} 
		\le \frac{\Gamma_{F}^{1/2}}{\Var F},\\
		\gamma_{6} &\leq \frac{2\sqrt{6}}{\Var F}
		\left[\int_{\XX^2} c_{x_1}^{4 /\left(4+q\right)} 
		\Prob{D_{x_{1}, x_{2}}^{2} F \neq 0}^{q /\left(8+2 q\right)} 
		\nu^{2}(\ldiff (x_{1}, x_{2}))\right]^{1 / 2} \\
		& \qquad\qquad\qquad\qquad +\frac{4\sqrt{3}}{\Var F}
		\left[\int_{\XX^2} c_{x_1}^{4 /\left(4+q\right)} 
		\Prob{D_{x_{1}, x_{2}}^{2} F \neq 0}^{q /\left(4+q\right)} 
		\nu^{2}(\ldiff(x_{1}, x_{2}))\right]^{1 / 2}\\
		&\le \frac{2\sqrt{6}+4\sqrt{3}}{\Var F}
		\left[\int_{\XX^2} c_{x_1}^{4 /\left(4+q\right)} 
		\Prob{D_{x_{1}, x_{2}}^{2} F \neq 0}^{q /\left(8+2 q\right)}
		\nu^{2}(\ldiff(x_{1}, x_{2}))\right]^{1 / 2}.
	\end{align*}
	Finally, by \cite[Lemma~4.3]{LPS16},
	\begin{align*}
		&\frac{\E(F-\E F)^{4}}{(\Var F)^{2}}\\
		&\leq \max \left\{\frac{256}{(\Var F)^2}
		\left[\int_\XX \left[\E\left(D_{x} F\right)^{4}\right]^{1 / 2} 
		\nu(\ldiff x)\right]^{2}, \frac{4}{(\Var F)^2} 
		\int_\XX \E\left(D_{x} F\right)^{4} \nu(\ldiff x)+2\right\}\\
		&\leq \max \left\{256  \Gamma_{F}^{2} /(\Var F)^{2}, 4  \Gamma_{F} /(\Var F)^{2}+2\right\},
	\end{align*}
	so that
	\begin{displaymath}
		\gamma_{4} \leq \frac{1}{(\Var F)^{3 / 2}} \Gamma_{F}
		+\frac{1}{(\Var F)^{2}} \Gamma_{F}^{5 / 4}+\frac{2}{(\Var F)^{2}} \Gamma_{F}^{3 / 2}.
	\end{displaymath}
	An application of \cite[Theorems~1.1 and 1.2]{LPS16} yields the results.
\end{proof}

\end{document}